\documentclass[reqno, 12pt, reqno]{amsart}

\usepackage{amsfonts, amsthm, amsmath, amssymb}
\usepackage{hyperref}
\hypersetup{colorlinks=false}
\usepackage{comment}

\usepackage[margin=3.2cm]{geometry}

\RequirePackage{mathrsfs}
\let\mathcal\mathscr

\numberwithin{equation}{section}

\newtheorem{theorem}{Theorem}[section]

\newtheorem{lemma}[theorem]{Lemma}

\newtheorem{corollary}[theorem]{Corollary}
\newtheorem{proposition}[theorem]{Proposition}

\theoremstyle{definition}
\newtheorem{definition}[theorem]{Definition}

\newtheorem{remark}[theorem]{Remark}
\newtheorem{remarks}[theorem]{Remarks}

\newtheorem*{remark*}{Remark}
\newtheorem*{remarks*}{Remarks}

\newcommand{\op}[1]{\operatorname{#1}}
\renewcommand{\l}{\left}
\renewcommand{\r}{\right}
\renewcommand{\ss}{\substack}
\newcommand{\f}{\frac}
\newcommand{\ra}{\rightarrow}
\newcommand{\eps}{\varepsilon}

\newcommand{\BA}{{\mathbb {A}}}

\newcommand{\BC}{{\mathbb {C}}}

\newcommand{\BF}{{\mathbb {F}}}
\newcommand{\BG}{{\mathbb {G}}}

\newcommand{\BN}{{\mathbb {N}}}

\newcommand{\BP}{{\mathbb {P}}}
\newcommand{\BQ}{{\mathbb {Q}}}
\newcommand{\BR}{{\mathbb {R}}}

\newcommand{\BZ}{{\mathbb {Z}}}

\newcommand{\CD}{{\mathcal {D}}}
\newcommand{\CE}{{\mathcal {E}}}
\newcommand{\CF}{{\mathcal {F}}}
\newcommand{\CG}{{\mathcal {G}}}
\newcommand{\CH}{{\mathcal {H}}}
\newcommand{\CI}{{\mathcal {I}}}
\newcommand{\CJ}{{\mathcal {J}}}

\newcommand{\CM}{{\mathcal {M}}}
\newcommand{\CN}{{\mathcal {N}}}
\newcommand{\CO}{{\mathcal {O}}}

\newcommand{\CS}{{\mathcal {S}}}
\newcommand{\CT}{{\mathcal {T}}}
\newcommand{\CU}{{\mathcal {U}}}
\newcommand{\CV}{{\mathcal {V}}}
\newcommand{\CW}{{\mathcal {W}}}

\newcommand{\bb}{{\mathbf{b}}}
\newcommand{\cc}{{\mathbf{c}}}
\newcommand{\dd}{{\mathbf{d}}}

\DeclareMathOperator{\bfx}{\mathbf{x}}

\DeclareMathOperator{\bft}{\mathbf{t}}
\DeclareMathOperator{\lcm}{lcm}

\renewcommand{\phi}{\varphi}
\renewcommand{\rho}{\varrho}
\renewcommand{\epsilon}{\varepsilon}

\renewcommand{\leq}{\leqslant}

\renewcommand{\geq}{\geqslant}

\newcommand{\uu}{\mathbf{u}}

\newcommand{\bx}{\boldsymbol{x}}
\newcommand{\by}{\boldsymbol{y}}
\newcommand{\bz}{\boldsymbol{z}}
\newcommand{\bt}{\boldsymbol{t}}
\newcommand{\bs}{\boldsymbol{s}}

\newcommand{\bc}{\boldsymbol{c}}
\newcommand{\bd}{\boldsymbol{d}}
\newcommand{\be}{\boldsymbol{e}}
\newcommand{\bu}{\boldsymbol{u}}
\newcommand{\bw}{\boldsymbol{w}}
\newcommand{\bv}{\boldsymbol{v}}

\newcommand{\blambda}{\boldsymbol{\lambda}}
\newcommand{\bLambda}{\boldsymbol{\Lambda}}
\newcommand{\bGamma}{\boldsymbol{\Gamma}}
\newcommand{\bsigma}{\boldsymbol{\sigma}}
\newcommand{\bnu}{\boldsymbol{\nu}}
\newcommand{\balpha}{\boldsymbol{\alpha}}
\newcommand{\bxi}{\boldsymbol{\xi}}

\newcommand{\e}{\textup{e}}
\newcommand{\aess}{\alpha_{\operatorname{ess}}}

\title{Optimal quantitative weak approximation for projective quadrics}
\author{Zhizhong Huang \and Damaris Schindler \and Alec Shute}
\date{\today}

\address{Institute of Mathematics, Academy of Mathematics and Systems Science, Chinese Academy of Sciences, Beijing, 100190, China}
\email{zhizhong.huang@yahoo.com}

\address{Mathematisches Institut, Georg-August-Universit\"{a}t G\"{o}ttingen, Bunsenstrasse 3--5,  37073, G\"{o}ttingen, Germany}
\email{damaris.schindler@mathematik.uni-goettingen.de}

\address{School of Mathematics, University of Bristol, Woodland Road, Bristol, UK}
\email{alec.shute@bristol.ac.uk}

\begin{document}
	\begin{abstract}
		We derive asymptotic formulas for the number of rational points on a smooth projective quadratic hypersurface of dimension at least three inside of a shrinking adelic open neighbourhood. This is a quantitative version of weak approximation for quadrics and allows us to deduce the best growth rate of the size of such an adelic neighbourhood for which equidistribution is preserved.
	\end{abstract}

\maketitle
\tableofcontents

	\section{Introduction}

The main goal of this article is to study the quantitative distribution of rational points on quadrics over the rational numbers. More precisely, for $n\geq 2$ we consider a quadratic form $F(x_0,\ldots, x_n)\in \mathbb{Z}[x_0,\ldots, x_n]$ in $n+1$ variables over the integers and ask for arithmetic properties of the variety $V\subseteq \mathbb{P}_{\mathbb{Q}}^n$ given by $F=0$. From now on we assume that the quadratic form $F$ has full rank, i.e. the variety $V$ is non-singular. It is well known by the celebrated Hasse-Minkowski theorem that $V$ satisfies the Hasse principle and weak approximation.\par
Weak approximation for $V\subseteq \mathbb{P}_{\mathbb{Q}}^n$ implies in particular that the set of rational points is infinite, if there are adelic points. In order to study the density of rational points on $V$, we may introduce a height function and count rational points on $V(\mathbb{Q})$ of height bounded by $B$ where $B$ tends to infinity. The behavior of this counting function is well studied and the asymptotic behavior is compatible with Manin's conjecture for rational points on Fano varieties. The goal of this article is to study a quantitative form of weak approximation for quadrics. More precisely we are interested in counting rational points on $V$ in a shrinking open neighborhood of the adelic points $V(\mathbf{A}_\BQ)$, and investigate how quickly the neighborhood can shrink when considering points of height bounded by $B$ with $B$ tending to infinity. With this our result fits into the general framework of Manin's conjecture on counting rational points of bounded height on Fano varieties, equidistribution questions for rational points on varieties and in particular questions on quantitative weak approximation.\par
More generally, let $X$ be a smooth projective variety over $\BQ$. The density of the set of rational points $X(\BQ)$ is one of the central themes of Diophantine geometry. To exclude the trivial case we shall assume $X$ is everywhere locally soluble, i.e., the \emph{adelic space} $X(\mathbf{A}_\BQ):=X(\BR)\times \prod_{p<\infty} X(\BQ_p)$ is non-empty. We say that $X$ satisfies \emph{weak approximation} if $X(\BQ)$ is dense in $X(\mathbf{A}_\BQ)$. See \cite{C-T} for a general conjecture of Colliot-Thélène concerning rationally connected varieties.
Weak approximation for $X$ implies in particular that $X(\BQ)$ is dense in the Zariski topology.
 An effective way of measuring this property is to attach $H:X(\BQ)\to \BR_{>0}$ a height function and to study asymptotic growth of points of bounded height. In the setting of Fano varieties, Manin's conjecture \cite{Bat-Man,FMT} provides a broad framework for the asymptotic behavior of points of bounded height outside of an exceptional subset. Refining Manin's conjecture, a general equidistribution principle proposed by Peyre \cite{Peyre} predicts that, if $X$ is Fano and $X$ satisfies weak approximation, upon removing some exceptional set $\CM\subseteq X(\BQ)$, the set $\{P\in X(\BQ)\setminus\CM:H(P)\leqslant B\}$, as $B\to\infty$, becomes equidistributed in $X(\mathbf{A}_\BQ)$. The question which exceptional set of rational points one has to remove in Manin's conjecture has received much attention during the last years \cite{Pey03,LST thin sets}.

 The Hardy-Littlewood circle method has proved to be a particularly flexible and useful tool to study both Manin's conjecture as well as Peyre's predicition on equidistribution for rational points. This has for example been executed for smooth hypersurfaces and more generally smooth complete intersections in projective space of sufficiently large dimension compared to the degree in work of Birch \cite{Birch} and Browning and Heath-Brown \cite{BHBdiffering}, as well as for certain hypersurfaces in toric varieties \cite{Schindler, Mignot, mignot2, PieropanSchindler}. It is an ongoing challenge to establish results of this type assuming only a smaller number of variables, for some improvements in this direction see for example work of Browning and Prendiville \cite{Browning-Prendiville} and work of Myerson \cite{Myerson1,Myerson2}.\par

One can ask similar questions to the above for affine varieties and integral points on affine varieties. For work concerning integral points on affine varieties, see for example \cite{Borovoi-Rudnick}  for symmetric varieties and \cite{EMV} for Linnik-type problems. In another direction Sardari \cite{Sardari} studied quantitative strong approximation questions for affine quadrics. Sardari's result \cite[Theorem 1.2]{Sardari} concerns Linnik-type problems about the family of affine quadrics $(F=N)\subseteq \BA^{n+1}$ indexed by a growing integer parameter $N$, where the condition at the real place is imposed via projecting onto the unit affine quadric defined by $F=1$. Our goal in this article is to study quantitative approximation questions for projective quadrics, where the condition at the real place is placed in the spirit of weak approximation at the real place of the single fixed projective quadric $(F=0)\subseteq\BP^{n}$. Similarly, Sardari and us handle these questions for quantitative strong and weak approximation conditions at non-archimedean places, but for simplicity of exposition we restrict for a moment to the real place.\par

Our general strategy for our main result is to first pass to the affine cone and apply the circle method to count integral points of height bounded by say $B$ on the affine cone, which are at the same time restricted to a neighborhood of radius $R^{-1}$ around a given real point. We consider $B$ and $R$ as two real parameters which both tend to infinity and for which $R< B^{\frac{1}{2}-\epsilon}$. In passing back from the affine cone to the projective setting, we use the standard procedure of Moebius inversion. However, this poses major new challenges in our situation. In order to apply Moebius inversion, we also need to understand the number of integral points on the affine cone of height bounded by $B'$ where $B'$ is any real number less than $B$ and where the points are still restricted to neighborhoods of size $R^{-1}$ for the same parameter $R$ which we started with, i.e. we cannot assume that $R< B'^{\frac{1}{2}-\epsilon}$ for these choices of heights $B'$. To solve this problem, we consider three ranges of parameters of heights $B'$. In the first range where $B'$ is relatively large, we apply the circle method to find actual asymptotics for our counting function. In the range where $B'$ is still somewhat large but significantly smaller than $B$ we obtain upper bounds for the corresponding counting function again using the delta method. In the last range for $B'$ relatively small, we complement our earlier results with lattice point counting techniques for upper bounds.\par

Our results on quantitative weak approximation  have applications in the setting of Diophantine approximation on algebraic varieties. In particular our results imply upper bounds for the \emph{approximation constants} as introduced by work of McKinnon--M. Roth \cite{M-R} for projective quadrics. We discuss these applications and connections in the appendix of this article.\par

	\subsection{Main counting result for the projective quadric}
	We now come to set up the terminology to state our results precisely. For this let $n\geq 3$ and $F\in\BZ[X_0,\ldots,X_n]$ be a non-degenerate quadratic form in $n+1$ variables which defines a smooth projective hypersurface $V=(F=0)\subseteq\BP_\BQ^n$ over $\BQ$, with the integral model $\CV\subseteq\BP^n_{\BZ}$ also defined by $F$.
	
	Let us first prescribe several local conditions in more detail which allow the size of adelic sets to shrink.
	We start by the one for the real place. 	Let us fix a point $\bxi \in V(\BR)$, and fix $g_{\bxi}:V(\BR)\dashrightarrow\BR^{n-1}$ a local diffeomorphism which sends $\bxi$ to $\boldsymbol{0}$. We also fix an $(n-1)$-dimensional bounded real neighbourhood of $\boldsymbol{0}$ in $\BR^{n-1}$, denoted by $U_0$. For a real parameter $R\geq 1$, the family 
	\begin{equation}\label{real zoom condition in x}
		\left(\CU(R,\bxi):=g_{\bxi}^{-1}(R^{-1}U_0)\right)_{R\geqslant 1}
	\end{equation} consisting of real neighbourhoods ``of diameter $R^{-1}$'' forms a topological basis for the real point $\bxi$ inside of $V(\BR)$. Consequently, as $R\to\infty$ the \textit{real zoom condition}  $\CU(R,\bxi)\cap V(\BQ)$ captures points of $V(\BQ)$ approximating $\bxi$ in the real locus.

	We next describe zoom conditions for the non-archimedean places. 
	Let $L>0$ be an integer. For every prime $p\mid L$, we let $\bLambda_p \in\CV(\BZ_p)$ and $\bLambda:=(\bLambda_p)_{p\mid L}$. Write $m_p:=\operatorname{ord}_p(L)$ and define
	$$\CD_p(L,\bLambda_p):=\{x_p\in\CV(\BZ_p): x_p\equiv \bLambda_p\bmod p^{m_p}\}.$$
	If $m_p\to\infty$ then points in $V(\BQ)\cap \CD_p(L,\bLambda_p)$ $p$-adically approximate $\bLambda_p$, and so are said to satisfy a \emph{$p$-adic zoom condition}.
	 The set $$\CD(L,\bLambda):=\prod_{p\mid L} \CD_p(L,\bLambda_p)\times\prod_{p\nmid L} V(\BQ_p)\subseteq \prod_{p<\infty} V(\BQ_p)$$ is a non-empty finite adelic open neighbourhood of $V$. 

 The set $\CU(R,\bxi)\times \CD(L,\bLambda)\subseteq V(\mathbf{A}_\BQ)$ is a non-empty open adelic set. We fix a height function $H_{\bxi}:\BP^n(\BQ)\to\BR_{>0}$ relative to the line bundle $\CO(1)$ which depends on the real point $\bxi$.
We want to estimate the number of rational points of $V$ of bounded height that ``lie in $U_0$ after scaling by $R$'' and simultaneously ``specialise to $\boldsymbol{\Lambda}$ modulo $L$''.  So we consider the counting function
\begin{equation}\label{eq:CNV}
		\CN_{V}((R,\bxi),(L,\bLambda);B):=\#\{P\in V(\BQ)\cap (\CU(R,\bxi)\times \CD(L,\bLambda)):H_{\bxi}(P)\leqslant B\}.
\end{equation}
	\begin{theorem}\label{thm:countingV} 

Assume $n\geqslant 4$ and $V(\mathbf{A}_\BQ)\neq\varnothing$. Let $H_{\bxi},g_{\bxi},U_0$ be defined by \eqref{eq:heightbxi} \eqref{eq:gxi} \eqref{eq:U0} respectively in Section \ref{se:realplace}. Assume there exists $0<\tau<1$ such that as $B\to\infty$, \begin{equation}\label{eq:tau}
			(LR)^2=O(B^{1-\tau}),
		\end{equation}
  and assume that $R\geq R_0$, where $R_0$ is defined in \eqref{sizeR0}. Then 
		$$\CN_{V}((R,\bxi),(L,\bLambda);B)=\frac{B^{n-1}}{n-1}\omega^V_{\bxi}\left(\CU(R,\bxi)\times \CD(L,\bLambda)\right)\left(1+ O_{\eps,\tau}(B^\varepsilon E_\tau(B))\right).
 $$
 
  Here $\omega^V_{\bxi}$ is the Tamagawa measure (à la Peyre \cite[Définition 2.2]{Peyre}) on $V(\mathbf{A}_\BQ)$ corresponding the height function $H_{\bxi}$, and 
\begin{equation}\label{EB}
			E_\tau(B) =
			\begin{cases}
				B^{\f{-(n-3)\tau}{5n-7}}, &\textrm{ if }n\geq 5;\\
				B^{-2\tau/17}, &\textrm{ if }n=4. 
			\end{cases}
		\end{equation} 
The implied constant depends at most on $\bxi$ (apart from $\varepsilon,\tau$) and is otherwise uniform with respect to $R,L$ and $\bLambda$.
	\end{theorem}
	
	Theorem \ref{thm:countingV}  implies a quantitative version of the classical Hasse--Minkowski theorem on weak approximation for projective quadrics, allowing the size of the adelic open set (i.e. the parameters $R,L$) to ``shrink'' in terms of $B$.  Indeed (cf. Proposition \ref{prop:growthofcirsigma}), $$L^{-\dim V-\varepsilon} R^{-\dim V}\ll_{\varepsilon}\omega^V_{\bxi}\left(\CU(R,\bxi)\times \CD(L,\bLambda)\right)\ll (LR)^{-\dim V},$$ and Theorem \ref{thm:countingV} furnishes an error term with a power-saving, subject to the condition \eqref{eq:tau} on the growth of $R,L$. 
 
  The main term in Theorem \ref{thm:countingV} can be equally be expressed as a product of local densities, whose form is well-known and can be achieved in different ways (see e.g. \cite[\S5]{Peyre}). However, the key characteristic of Theorem \ref{thm:countingV} is that the condition \eqref{eq:tau} is \emph{optimal}, in the sense that such an equidistribution can fail if  the growth rate of $R$ (resp. $L$) exceeds a power of $B$ strictly greater than $\frac{1}{2}$, without further assumption on the Diophantine properties of $\bxi$ (resp. $\bLambda$). We refer to the Appendix for a short discussion in view of Diophantine approximation, notably Remarks \ref{rmk:dioapp} (1).

	\subsection{Main counting result for the affine cone}\label{se:settingandmaincounting}	
	Let $W\subseteq \BA^{n+1}$ be the affine cone of $V$ with the integral model $\CW\subseteq \BA^{n+1}_{\BZ}$ defined by $F$. 
   Let $W^o\subseteq \BA^{n+1}\setminus\boldsymbol{0}$ be the punctured affine cone over $V$. We take the integral model $\CW^o:=\CW\setminus \overline{\boldsymbol{0}}$ (where $\overline{\boldsymbol{0}}$ stands for the Zariski closure of $\boldsymbol{0}$ in $\BA^{n+1}_{\BZ}$). Let $\pi:\BA^{n+1}\setminus\boldsymbol{0}\to \BP^n$ be the canonical projection map.
   
 Let $\bGamma\in(\BZ/L\BZ)^{n+1}$ be such that $F(\bGamma)\equiv 0\bmod L$. We can view $\bGamma$ as an element of $\CW(\BZ/L\BZ)$. Let   $\|\cdot\|_{\bxi}:\BR^{n+1}\to\BR_{\geqslant 0}$ be the norm in accordance with $H_{\bxi}$. 
 The counting function \eqref{eq:CNV} is closely related to
	\begin{multline}\label{eq:CNW}
		\CN_{\CW}((R,\bxi),(L,\bGamma);B) \\:= \#\l\{\bx \in \BZ^{n+1}\setminus\boldsymbol{0}: F(\bx) = 0, \bx \equiv \bGamma \bmod{L}, \pi(\bx)\in \CU(R,\bxi), \|\bx\|_{\bxi} \leq B \r\}.
	\end{multline}
	We now state our main counting result about \eqref{eq:CNW}. It features a \textit{singular series} \begin{equation}\label{eq:singserW}
		\mathfrak{S}(\CW;L,\bGamma):=\prod_p\sigma_{p}(\CW;L,\bGamma)
	\end{equation} and a  \textit{singular integral} $\CI_R$, which are given by the equations (recall that $m_p:=\operatorname{ord}_p(L)$) 
	\begin{align}
		\sigma_{p}(\CW;L,\bGamma)&:=\lim_{k\to\infty}\frac{\#\{\bv\in \CW(\BZ/p^k\BZ):\bv\equiv\bGamma\bmod p^{m_p}\}}{p^{nk}},\label{eq:sigmapW}\\
		\CI_R &:= \int_{\theta \in \BR}\int_{\ss{\bx \in \BR^{n+1}: \|\bx\|_{\bxi}\leqslant 1\\ \pi(\bx)\in \CU(R, \bxi)}}\e(\theta F(\bx)) \op{d}\bx \op{d}\theta.\label{eq:singint}
	\end{align}
	
	\begin{theorem}\label{main counting result} 

Let $\|\cdot\|_{\bxi}$ be defined by \eqref{eq:normbxi}. Assume that $\bGamma\in\CW^o(\BZ/L\BZ)$ and that the same assumptions as in Theorem \ref{thm:countingV} hold. Then we have
		$$ \CN_{\CW}((R,\bxi),(L,\bGamma);B)= B^{n-1}\CI_R\mathfrak{S}(\CW;L,\bGamma)\l(1 + O_{\eps,\tau}(B^\varepsilon E_\tau(B))\r),$$
		where $E_\tau(B)$ is given in (\ref{EB}). Here again the implied constant may depend on $\bxi$ (apart from $\varepsilon,\tau$), but is otherwise uniform with respect to $R,L$ and $\bGamma$. 

  \end{theorem}

 The case of quaternary quadratic forms (i.e. $n=3$) is more subtle and is addressed in a separate article \cite{PART2}.

 \subsection{Comparisons with results in the literature and further comments}
 Similar forms of Theorems \ref{thm:countingV} \& \ref{main counting result} have previously been obtained by the work of Browning--Loughran in \cite[\S4]{Browning-Loughran} and by Kelmer--Yu in \cite[\S1.3]{Kelmer-Yu}. 
 
In \cite{Browning-Loughran} the slightly different problem of counting smoothly weighted solutions of $F$ whose reduction modulo $L$ lies in a given subset of local solutions is concerned (without prescribing any real zoom condition). In this setting, the expected order of growth is $B^{n-1}$ and $L^{-n}$ is roughly the ``volume'' of the affine adelic open set defined by the congruence condition $\bx \equiv \bGamma \bmod{L}$. Therefore, it can be checked that \cite[Theorem 4.1]{Browning-Loughran} produces an asymptotic formula provided that 
 $$ L^{\frac{n+1}{2}}B^{\frac{n+1}{2}+\eps}  = o(B^{n-1}L^{-n}),$$
 or in other words, $L = O(B^{\delta})$ for some $\delta <\frac{n-3}{3n+1}=\frac{1}{3}-\frac{10}{3(3n+1)}$. (It is also suggested therein that when $L$ is squarefree, their error term could be improved, with the result that $L$ could grow like $B^{\delta}$ for $\delta < \frac{1}{2}- \frac{3}{2n}$. This becomes optimal in the limit as $n\to \infty$.) Theorem \ref{main counting result} thus improves upon \cite[Theorem 4.1]{Browning-Loughran} in the $L$ aspect, since it allows $L$ to grow arbitrarily close to $B^{\frac{1}{2}}$ for any $n\geq 4$. As explained before, $B^\frac{1}{2}$ is the ``barrier'' for which such an equidistribution is preserved. Also here we do not require the assumption made in  \cite[Theorem 4.1]{Browning-Loughran} that $L$ is coprime to $2\Delta_F$.  

 In \cite{Kelmer-Yu} however, via a quite different approach and as an application of its main result, \cite[Theorem 1.7]{Kelmer-Yu} deals with the same counting problem for the quadric $V$ as ours, but restricted to quadratic forms of signature $(n,1)$, for which only real zoom condition is imposed (i.e. $L=1$). The optimal growth range for $R$ is achieved only for a sub-family of quadrics (see \cite[Remark 1.3]{Kelmer-Yu}). So Theorem \ref{thm:countingV} generalises and improves upon \cite[Theorem 1.7]{Kelmer-Yu} not only covering quadrics with any signture but also in the $R$-aspect.

We would like to comment on the optimality of Theorem \ref{thm:countingV}.
In the light of the Schmidt subspace theorem and its ultrametric generalisations, it is probable that, upon specifying finer Diophantine approximation properties of $\bxi$ or $\bLambda$, the asymptotic in Theorem \ref{thm:countingV} would still remain true under weaker constraints on the growth of $R$ or $L$. See also Remarks \ref{rmk:dioapp} (2).
However, achieving this seems to go beyond the technical limit developed in this paper (and also the circle method of Birch's type \cite{Birch}) -- we are unable to keep track of any intrinsic approximation properties of either $\bxi$ or $\bLambda$.

	Our approach is based upon a $\delta$-variant of the Hardy--Littlewood circle method which goes back to Duke--Friendlander--Iwaniec \cite{D-F-I} and is later extensively developed by Heath-Brown \cite{H-Bdelta}. adapting particularly well to quadratic forms. The execution of the $\delta$-method leads to a certain type of twisted quadratic exponential sums, similar to the globally defined $S_q(\bc)$ in \cite{H-Bdelta}.
 However, these exponential sums depend on the newly introduced non-archimedean local condition $(L,\bLambda)$. To obtain satisfactory control for them we work out certain refinements of \cite[\S9--\S11]{H-Bdelta}, and we need to take special care of uniformity, especially over the ``bad moduli''. \par
 
	An important feature in Sardari's work \cite{Sardari} is a careful construction of smooth weight functions to capture the Linnik-type real condition in the affine setting. 
	In the case of projective varieties, the quantitative weak approximation conditions lead to differently shaped regions on the affine cone and with this we need to build a different weight function encapsulating as many ``level-lowering'' effects as possible that these conditions result in, which takes an equally crucial role for us as it did in \cite{Sardari}. Furthermore, the application of the $\delta$-method requires smoothed counting functions and in order to deduce an asymptotic for our original counting function, we need to apply a removing of smooth weights procedure which introduces some additional technicalities which we also discuss in detail in this article.\par

	\subsection{Organisation of the article}
In \S\ref{se:setupdelta}, we explicitly construct smooth weight functions, with respect to which we state a ``weighted'' counting result (Theorem \ref{thm:countingW}). The remaining part of \S\ref{se:setupdelta} discusses how the $\delta$-method is executed. In \S\ref{se:oscint} and \S\ref{se:quadexpsum} we prove  estimates for the oscillatory integrals and the singular series respectively. Sections \S\ref{se:cneq0} and \S\ref{se:c=0} are devoted to the proof of Theorem \ref{thm:countingW}. In \S\ref{se:removeweight} we prove Theorem \ref{main counting result} using Theorem \ref{thm:countingW} with the help of a removing of smooth weights procedure, and in \S\ref{se:proofofmainthmprojquadric} we deduce Theorem \ref{thm:countingV} from  Theorem \ref{main counting result}. Section \ref{se:proofmainthm} also contains upper bounds for the counting function $\CN_{\CW}((R,\bxi),(L,\bGamma);B)$ in the range $-1<\tau<1$ using the delta method, as well as upper bounds for all ranges of $B,L,R$ using lattice point counting techniques. 
 In the Appendix we discuss several aspects from Diophantine approximation. On combining with our main counting result, we study various approximation constants for projective quadrics.

\subsection{Notation and conventions} Unless otherwise specified, all implied constants are allowed to depend on the quadratic form $F$ (and hence the quadric $V$), the real point $\bxi\in V(\BR)$ and the real neighbourhood $U_0$, which are considered fixed throughout. 

We write $\tau(n)$ for the divisor function.
We write $\Delta_F$ for the discriminant of the quadratic form $F$. For non-zero integers $k,A$, we use the notation $k \mid A^{\infty}$ to mean that every prime factor of $k$ is also a prime factor of $A$. 
The symbol $T\nearrow X$ within a sum means a dyadic decomposition, i.e. $T$ runs over all powers of $2^j$ with $2^j\leq X$.

	\section{Set-up for the $\delta$-method}\label{se:setupdelta}
	\subsection{Considerations for the real place}\label{se:realplace}
	Without loss of generality, we may assume the point $\bxi= [\xi_0: \cdots :\xi_n] $ satisfies $\xi_0 \neq 0$.
	We begin by applying a (classical) change of variables in order to get a ``nice'' coordinate system. 
	
	\begin{lemma}\label{change of variables}
		There exists a matrix $M \in \op{GL}_{n+1}(\BR)$, sending the variables $\bx$ to $\bt := M\bx$, such that 
		\begin{itemize}
			\item $M\bxi = [1:0\cdots :0]$;
			\item $F(\bx) = F(M^{-1}\bt) =t_0t_1 + F_2(t_2, \ldots, t_n)$, where $F_2$ is a non-degenerate quadratic form in $(n-1)$ variables.
		\end{itemize}
	\end{lemma}
	
	\begin{proof}
		Via a slight abuse of notation we write $\bxi = (\xi_0,\ldots, \xi_n)$ and recall that we have $F(\bxi)=0$. Let $B(\bu,\bv)$ be the bilinear form associated to the quadratic form $F$. Choose a vector $\bv_1\in \BR^{n+1}$ which is linearly independent of $\bxi$ and satisfies $B(\bv_1,\bxi)\neq 0$. For $\lambda \in \BR$ we have
  $$F(\bv_1+\lambda \bxi)= F(\bv_1) + 2 \lambda B(\bv_1, \bxi).$$
  Hence we can find $\lambda$ such that $F(\bv_1+\lambda \bxi)=0$ and $B(\bv_1+\lambda \bxi,\bxi)\neq 0$. We replace $\bv_1$ by $\bv_1+\lambda \bxi$. Moreover, after multiplying $\bv_1$ by a scalar we may assume that $2B(\bv_1,\bxi)=1$. Next consider the orthogonal complement of $\bv_1$ and $\bxi$ under the bilinear form $B$ and choose a basis $\bv_2,\ldots, \bv_n$. In particular we have
  $$B(\bv_j,\bv_1)=B(\bv_j,\bxi)=0,\quad 2\leq j\leq n,$$
  and the vectors $\bxi,\bv_1,\ldots, \bv_n$ are a basis of $\BR^{n+1}$. Now we take $M$ to be the matrix such that the columns of $M^{-1}$ are given by $\bxi,\bv_1,\ldots, \bv_n$.
\end{proof}
With respect to this new coordinate system $\bt$, we now define the norm $\|\cdot\|_{\bxi}$ and the height function $H_{\bxi}$ as follows:
 \begin{equation}\label{eq:normbxi}
     \|\bx\|_{\bxi}:=\max_{0\leqslant i\leqslant n}|t_i(\bx)|,\quad  \text{for }\bx\in\BR^{n+1},
 \end{equation}
	\begin{equation}\label{eq:heightbxi}
	    H_{\bxi}([x_0:\cdots:x_n]):=\|\bx\|_{\bxi},\quad \text{if }\bx=(x_0,\ldots,x_n)\in\BZ^{n+1}\text{ and }\gcd_{0\leq i\leq n}(x_i)=1.
	\end{equation} Hence $H_{\bxi}$ is the naive height function after a linear change of coordinate systems.
	
	Applying the change of variables from Lemma \ref{change of variables}, we see that in the chart $t_0\neq 0$, the affine tangent plane at the point $\boldsymbol{0}$ is $T_{\boldsymbol{0}} V(\BR)=(\frac{t_1}{t_0}=0)$. Moreover,  $\frac{t_2}{t_0},\ldots,\frac{t_n}{t_0}$ form a local coordinate system of the hypersurface $V$ around $\bxi$. We now take the local diffeomorphism as 
 \begin{equation}\label{eq:gxi}
     g_{\bxi}(\pi(\bx)):=\left(\frac{t_2}{t_0}(\bx),\ldots,\frac{t_n}{t_0}(\bx)\right).
 \end{equation}
 
	By covering any open neighborhood of $\mathbf{0}$ by cubes, it is sufficient to consider the case \begin{equation}\label{eq:U0}
	    U_0=\prod_{j=2}^{n}[a_j,b_j],
	\end{equation} a standard $(n-1)$-dimensional cube, where $-\infty<a_j<b_j<\infty$ for all $2\leqslant j\leqslant n$. The real zoom condition (\ref{real zoom condition in x}) becomes 
	\begin{equation}\label{real zoom condition in t}
		R\f{t_j}{t_0}(\bx)\in[a_j,b_j] \textrm{ for all }2\leq j \leq n.
	\end{equation}
	
	\subsection{Choice of weight functions}

	In order to apply the $\delta$-method of Heath-Brown, we need to work with a smoothly weighted analogue of $\CN_{\CW}((R,\bxi),(L,\bGamma);B)$. 
 For an infinitely differentiable function $w: \BR^{n+1}\ra \BR_{\geq 0}$, we define 
	
	\begin{equation}\label{smoothly weighted counting problem}
		\CN_{\CW}(w;(L,\bGamma)):=\sum_{\substack{\bx\in\BZ^{n+1}\\ F(\bx)=0, \bx\equiv \bGamma\bmod L}}w(\bx).
	\end{equation}
	We would like $w$ to capture the real zoom condition (\ref{real zoom condition in t}) and the bounded height condition. Therefore, we let
	
	\begin{equation}\label{general form of weight function wbr}
		w_{B,R}(\bx)=w_1\left(R\frac{t_j}{t_0}(\bx),j\geqslant 2\right)w_2\left(\frac{R^2}{B^2}F(\bx)\right)w_3\left(\frac{t_0(\bx)}{B}\right),
	\end{equation}
	where $w_1:\BR^{n-1}\to [0,1], w_2, w_3:\BR\to[0,1]$ are infinitely differentiable weight functions and $w_2(0)=1$. 
 
 Several comments on the role of $w_1,w_2,w_3$ are in order. Clearly $w_1$ ``mollifies'' the neighbourhood $U_0$. Due to (\ref{real zoom condition in t}) and the assumption $R\geq 1$, the condition $\|\bx\|_{\bxi}\ll B$ is equivalent to $|t_0| \ll B$. Therefore, $w_3$ imposes a ``mollified bounded height condition''.
 The inclusion of the weight function $w_2$ is crucial for the ``level-lowering''. Indeed, let $\bx\in\BZ^{n+1}\setminus\boldsymbol{ 0}$ be such that $t_0(\bx)\neq 0$. If \eqref{real zoom condition in t} is satisfied, as prescribed by $w_1$, then (as in Lemma \ref{change of variables}) $F_2(t_2, \ldots, t_n) \ll R^{-2}|t_0|^2$, and so the equation $$F(\bx)=t_0^2\left(\frac{t_1}{t_0} + F_2\left(\frac{t_2}{t_0}, \ldots ,\frac{t_n}{t_0}\right)\right)=0$$
 implies that \begin{equation}\label{eq:t1t0}
	    |t_1/t_0| \ll R^{-2}.
	\end{equation} If moreover $\|\bx\|_{\bxi}\ll B$, as prescribed by $w_3$, then together with \eqref{real zoom condition in t} and \eqref{eq:t1t0} we obtain $$F(\bx)\ll \frac{B^2}{R^2}.$$ 
 Since this bound relies on the condition $F(\bx) = 0$ (which will be captured later using the $\delta$-symbol), the effect of $w_2$ cannot be encapsulated by $w_1$ and $w_3$ alone. We could have instead inserted a weight function of the form $w_4\left(R^2\frac{t_1}{t_0}\right)$, but this turns out to have an equivalent role to $w_2$, and working with $w_2$ is technically simpler. 
 
 Moreover, to deduce information about the ``unweighted'' counting function $\CN_{\CW}((R,\bxi),(L,\bGamma);B)$ from the weighted one $\CN_{\CW}(w_{B,R};(L,\bGamma))$, we would like $w_1$ and $w_3$ to approximate the characteristic functions $\prod_{j\geq 2} 1_{[a_j,b_j]}$ and $1_{[-1,1]}$ respectively. To make this more precise, we define a small parameter $0<\eta <1$ (which will eventually be taken to be a small negative power of $B$), and consider smooth weights $w_1^{\pm}:\BR^{n-1}\ra [0,1], w_3^{\pm}: \BR\ra [0,1]$ satisfying 
	\begin{enumerate}
		\item $\operatorname{supp}(w_1^+)=\prod_{j\geq 2}[a_j-\eta, b_j+\eta];$
		\item $w_1^+ = 1$ on $\prod_{j\geq 2}[a_j, b_j]$;
		\item $\operatorname{supp}(w_1^-)=\prod_{j\geq 2}[a_j, b_j];$
		\item $w_1^- =1$ on $\prod_{j\geq 2}[a_j+\eta, b_j-\eta]$; 
		\item $\operatorname{supp}(w_3^+) = [-1-\eta,1+\eta]\setminus [-\frac{\eta}{2},\frac{\eta}{2}]$;
		\item $w_3^+ = 1$ on $[-1,1]\setminus [-\eta,\eta]$;
		\item $\operatorname{supp}(w_3^-) = [-1,1]\setminus [-\frac{\eta}{2},\frac{\eta}{2}]$;
		\item $w_3^- = 1$ on $[-1+\eta,1-\eta]\setminus [-\eta,\eta]$.
	\end{enumerate}
	For an infinitely differentiable function $f:\BR^{m} \ra \BR$ and an integer $N\geq 0$, we define 
	\begin{equation}\label{def partial derivative norm}
 \|f\|_N = \sup_{\ss{\bx \in \BR^{m}\\ \boldsymbol{k} \in (\BZ_{\geq 0})^{m}\\ k_1+\cdots + k_m \leq N}}|\partial^{\boldsymbol{k}}f(\bx)|
 \end{equation}
	to be the largest partial derivative of $f$ of order at most $N$. In particular, we recover the supremum of $|f|$ via $\|f\|_0$, a fact which is used in the analysis of the oscillatory integrals.
	
	\begin{lemma}\label{existence of smooth weights}
		There exist smooth weight functions $w_1^{\pm}, w_3^{\pm}$ satisfying the above properties, and such that for any $N\geq 0$, we have $\|w_1^{\pm}\|_N, \|w_3^{\pm}\|_N \ll_N \eta^{-N}$.
	\end{lemma}
	\begin{proof}
		Weight functions with the desired properties can all be built from linear transformations and convolutions of the \textit{standard bump function} $\psi:\BR\rightarrow \BR_{\geq 0}$ defined by 
		\begin{equation*}
			\psi(x) =
			\begin{cases}
				\exp\l(1+\f{1}{|x|^2-1}\r), &\textrm{ if }|x|< 1,\\
				0, &\textrm{ otherwise}
			\end{cases}
		\end{equation*}
		and its higher dimensional analogue, in which $|\cdot|$ is replaced with the Euclidean norm (see \cite[Section 4]{Shute}).
	\end{proof}
	
	Let $w_2:\BR \ra \BR_{\geq 0}$ be the standard bump function, and let $w_1^{\pm}, w_3^{\pm}$ be as in Lemma \ref{existence of smooth weights}. We define weight functions $w_{B,R}^{\pm}: \mathbb{R}^{n+1}\rightarrow \mathbb{R}_{\geq 0}$ by 
	\begin{equation}\label{choice of wbr}
		w_{B,R}^{\pm}(\bx) := w_1^{\pm}\left(R\frac{t_j}{t_0}(\bx), j \geq 2\right)w_2\left(\frac{R^2}{B^2}F(\bx)\right)w_3^{\pm}\left(\frac{t_0(\bx)}{B}\right).
	\end{equation}
	
	It is clear that $w_{B,R}^{\pm} $ is supported away from zero, and satisfies $\|w_{B,R}^{\pm}\|_N \ll_N \eta^{-N}$ for any $N\geq 0$. We shall also make use of the functions $w_R^{\pm}$, defined as in (\ref{choice of wbr}) but without the $w_2$ and $w_3^{\pm}$ factors. 

 \subsection{Weighted counting results}
 We are ready to state our main counting results in terms of the aforementioned weight functions.
	\begin{theorem}\label{thm:countingW}

	Assume that $R,L$ satisfy \eqref{eq:tau} for $-1<\tau<1$. Let $w_{B,R}^{\pm}
 $ be one of the smooth weights defined by \eqref{choice of wbr} and $\varepsilon >0$. Let $A>0$ be a real parameter such that $\eta^{-1}\ll B^A$.
  Then uniformly in $L,\bGamma\in\CW^o(\BZ/L\BZ)$, we have  
		$$\CN_{\CW}(w_{B,R}^{\pm};(L,\bGamma))=B^{n-1}\CI(w_R^{\pm})\mathfrak{S}(\CW;L,\bGamma)+O_{\tau, A,\varepsilon}\left(\frac{B^{n-1+\varepsilon}}{L^n R^{n-1}}E_\tau(\eta;B)\right), $$ where \begin{equation}\label{eq:Etau}
		    E_\tau(\eta;B):= \begin{cases}
			\eta^{\f{9-5n}{2}}B^{-\frac{n-3}{2}\tau} &\text{ if } n\geqslant 5;\\
			\eta^{-15/2}B^{-\tau} + B^{-\frac{1+\tau}{4}} & \text{ if } n=4.
		\end{cases}
		\end{equation}
		Here $\mathfrak{S}(\CW;L,\bGamma)$ is the singular series defined in \eqref{eq:singserW}  and $\CI(w_R^{\pm})$ is a singular integral defined later in \eqref{eq:CIwR}.
	\end{theorem}
	
	\begin{remarks}\hfill
 \begin{itemize}
 \item If $0<\tau<1$, then the main term $B^{n-1}\CI(w_R^{\pm})\mathfrak{S}(\CW;L,\bGamma)$ is dominant and (up to the weight function $w$) eventually gives the main contribution in Theorem \ref{main counting result}. However, the ``error term'' becomes dominant when $-1<\tau\leqslant 0$. Even though in this case Theorem \ref{thm:countingW} provides only an upper bound, this is crucial for the removal of smooth weights back to the ``constant weight''.
     \item  In Theorem \ref{thm:countingW}, once $w_1,w_2,w_3$ are fixed, the implied constant depends only on the small parameter $\eta$. We may replace $w^{\pm}_{B,R}$ with any weight function $w$ of the form (\ref{general form of weight function wbr}) which is supported away from zero, at the cost of having implied constants that depend implicitly on the $w_1,w_2,w_3$-components of $w$.
 \end{itemize}
	\end{remarks}

	\subsection{Poisson summation}
	For $n\in\BZ$, we define the $\delta$-symbol to be $$\delta(n)=\begin{cases}
		1 & \text{ if } n=0;\\ 0 & \text{ otherwise}.
	\end{cases}$$
	According to Heath-Brown's work \cite[Theorem 1]{H-Bdelta}, for every $Q>1$, we can write
	\begin{equation}\label{eq:delta}
		\delta(n)=\frac{C_Q}{Q^2}\sum_{q=1}^{\infty}\sum_{\substack{a\bmod q\\(a,q)=1}}\textup{e}_q\left(an\right)h\left(\frac{q}{Q},\frac{n}{Q^2}\right),
	\end{equation} where $$C_Q=1+O_N(Q^{-N}),$$ and $h:(0,\infty)\times\BR\to\BR$ is a certain infinitely differentiable function, whose properties are described in \cite[\S4]{H-Bdelta}.
	
	Given $\bGamma\in\CW(\BZ/L\BZ)$, we shall work with a fixed representative $\blambda=(\lambda_0,\cdots,\lambda_n)\in\BZ^{n+1}$ of $\bGamma$ whose coordinates are bounded by $L$. Note that the assumption $\bGamma\in\CW^o(\BZ/L\BZ)$ implies that the gcd of the coordinates of $\blambda$ is coprime to $L$.
 Let us consider the polynomial 
	\begin{equation}\label{eq:H}
		H_{\boldsymbol{\lambda},L}(\by):=\frac{F(\boldsymbol{\lambda})}{L}+\nabla F(\boldsymbol{\lambda})\cdot \by.
	\end{equation} From now on, let $w_{B,R}$ be a weight function of the form \eqref{general form of weight function wbr}. The goal now is to show 
		\begin{proposition}\label{prop:deltacount}
			We have
		\begin{equation}\label{eq:Nwdelta}
			\begin{split}
				\CN_{\CW}(w_{B,R};(L,\bGamma))=\frac{C_Q}{Q^2}\sum_{q=1}^{\infty}\sum_{\bc\in\BZ^{n+1}}\frac{S_{q,L,\blambda}(\bc)I_{q,L,\blambda}(w_{B,R};\bc)}{(qL)^{n+1}}, 	\end{split}
		\end{equation} where $C_Q=1+O_N(Q^{-N})$ and for every $q\ll Q, \bc\in\BZ^{n+1}$,
	\begin{equation}\label{eq:Iqc1}
	I_{q,L,\blambda}(w_{B,R};\bc):=\int_{\BR^{n+1}}w_{B,R}(L\by+\boldsymbol{\lambda})h\left(\frac{q}{Q},\frac{F(L\by+\boldsymbol{\lambda})}{L^2Q^2}\right)\e_{qL}\left(-\bc\cdot\by\right)\operatorname{d}\by,
\end{equation} 	and \begin{equation}\label{eq:Sqc}
S_{q,L,\blambda}(\bc):=\sum_{\substack{a\bmod q\\(a,q)=1}}\sum_{\substack{\bsigma\in(\BZ/qL\BZ)^{n+1}\\ H_{\boldsymbol{\lambda},L}(\boldsymbol{\sigma})\equiv 0\bmod L}}\textup{e}_{qL}\left(a\left(H_{\boldsymbol{\lambda},L}(\boldsymbol{\sigma})+LF(\boldsymbol{\sigma})\right)+\bc\cdot\bsigma\right).
\end{equation}
	\end{proposition}
	Similarly to \cite{Browning-Loughran}, the proof of Theorem \ref{thm:countingW} relies on Heath-Brown's $\delta$-method \eqref{eq:delta}.
	However, to obtain the optimal error term, our way of activating \eqref{eq:delta} is different from \cite[\S4]{Browning-Loughran}, in which we need to understand properly the precise range of $F$ when zoom conditions are imposed and then to integrate it into the formula of \cite[Theorem 2]{H-Bdelta}. This is part of our ``level-lowering'' procedure regarding the congruence condition $(L,\blambda)$, and our treatment is closer to Sardari's \cite{Sardari}.

\begin{lemma}\label{le:deltacount}
	We have
	\begin{equation}\label{eq:N1}
		\CN_{\CW}(w_{B,R};(L,\bGamma))=\sum_{\substack{\by\in\BZ^{n+1}\\ L\mid H_{\boldsymbol{\lambda},L}(\by)}} w_{B,R}(L\by+\boldsymbol{\lambda})\delta\left(\frac{F(L\by+\boldsymbol{\lambda})}{L^2}\right).
	\end{equation}
\end{lemma} 
\begin{proof}
	We first observe that, if $\bx\equiv \boldsymbol{\lambda}\bmod L$, on writing $\bx=L\by+\boldsymbol{\lambda}$, we have $$F(\bx)=F(\boldsymbol{\lambda})+L\nabla F(\boldsymbol{\lambda})\cdot \by+L^2 F(\by)=LH_{\boldsymbol{\lambda},L}(\by)+L^2 F(\by).$$ We thus have $$L^2\mid F(L\by+\boldsymbol{\lambda})\Leftrightarrow L\mid H_{\boldsymbol{\lambda},L}(\by).$$ 	So  the expression above is well defined, and only captures $\by$ satisfying $F(L\by+\blambda)=0$.
\end{proof}

\begin{proof}[Proof of Proposition \ref{prop:deltacount}]

By Lemma \ref{le:deltacount} and \eqref{eq:delta}, the counting function $\CN_{\CW}(w_{B,R};(L,\bGamma))$ is equal to
\begin{equation}\label{eq:N2}
	\frac{C_Q}{Q^2}\sum_{q=1}^{\infty}\sum_{\substack{a\bmod q\\(a,q)=1}}\sum_{\substack{\by\in\BZ^{n+1}\\ L\mid H_{\boldsymbol{\lambda},L}(\by)}} w_{B,R}(L\by+\boldsymbol{\lambda})\textup{e}_{qL}\left(a\left(H_{\boldsymbol{\lambda},L}(\by)+LF(\by)\right)\right)h\left(\frac{q}{Q},\frac{F(L\by+\boldsymbol{\lambda})}{L^2Q^2}\right).
\end{equation}

	In order to apply Poisson summation and separate the real condition from the $p$-adic ones, we break the $\by$-sum into residue classes modulo $qL$, so that for every fixed $q,a$ and for every fixed class $\boldsymbol{\sigma}\bmod qL$ with $H_{\boldsymbol{\lambda},L}(\boldsymbol{\sigma})\equiv 0\mod L$, the inner $\by$-sum is equal to
	\begin{align*}
		&\textup{e}_{qL}\left(a\left(H_{\boldsymbol{\lambda},L}(\boldsymbol{\sigma})+LF(\boldsymbol{\sigma})\right)\right)\sum_{\substack{\by\in\BZ^{n+1}\\ \by\equiv \boldsymbol{\sigma}\bmod qL}}w_{B,R}(L\by+\boldsymbol{\lambda})h\left(\frac{q}{Q},\frac{F(L\by+\boldsymbol{\lambda})}{L^2Q^2}\right)\\ =& \textup{e}_{qL}\left(a\left(H_{\boldsymbol{\lambda},L}(\boldsymbol{\sigma})+LF(\boldsymbol{\sigma})\right)\right)	\sum_{\bz\in\BZ^{n+1}}w_{B,R}(L(\boldsymbol{\sigma}+qL\bz)+\boldsymbol{\lambda})h\left(\frac{q}{Q},\frac{F(L(\bsigma+qL\bz)+\boldsymbol{\lambda})}{L^2Q^2}\right).
	\end{align*}
	Applying the Poisson summation formula to the inner $\bz$-sum yields
	$$\sum_{\bc\in\BZ^{n+1}}\int_{\BR^{n+1}}w_{B,R}(L(\boldsymbol{\sigma}+qL\bz)+\boldsymbol{\lambda})h\left(\frac{q}{Q},\frac{F(L(\bsigma+qL\bz)+\boldsymbol{\lambda})}{L^2Q^2}\right)\e\left(-\bc\cdot\bz\right)\operatorname{d}\bz.$$ The inner integral, with the change of variables $\by=\bsigma+qL\bz$, is equal to
	$$\frac{\e_{qL}\left(\bc\cdot\bsigma\right)}{(qL)^{n+1}}I_{q,L,\blambda}(w_{B,R};\bc),$$  	
    where $I_{q,L,\blambda}(w_{B,R};\bc)$ is defined by \eqref{eq:Iqc1}.
    Hence the inner $\by$-sum in \eqref{eq:N2} is equal to $$\sum_{\substack{\bsigma\in(\BZ/qL\BZ)^{n+1}\\ H_{\boldsymbol{\lambda},L}(\boldsymbol{\sigma})\equiv 0\bmod L}}\sum_{\bc\in\BZ^{n+1}} \textup{e}_{qL}\left(a\left(H_{\boldsymbol{\lambda},L}(\boldsymbol{\sigma})+LF(\boldsymbol{\sigma})\right)+\bc\cdot\bsigma\right)\frac{I_{q,L,\blambda}(w_{B,R};\bc)}{(qL)^{n+1}}.$$
	Substituting into \eqref{eq:N2} yields the terms $S_{q,L,\blambda}(\bc)$ and we get the expression  \eqref{eq:Nwdelta}.
\end{proof}

	\subsection{Choice of $Q$}
	In terms of the zoom condition \eqref{real zoom condition in t}, we explain our choice of $Q$, keeping in mind its role in \eqref{eq:N1}, \eqref{eq:N2}  and \eqref{eq:Iqc1}.
 
	Recall from \cite[Theorem 1]{H-Bdelta} that $h(x,y)\neq 0$ only if $x\leqslant \max(1,2|y|)$.  We would like to have \begin{equation}\label{eq:chooseQ}
		\left|\frac{F(\bx)}{L^2}\right|\leqslant \frac{1}{2}Q^2,
	\end{equation} in order to deduce that $q\leqslant Q$. 
	We recall that $\bt=M \bx$, and by our choice of $M$ we have
	\begin{equation}\label{eq:tildeF}
		\widetilde{F}(\bt):=F(M^{-1}\bt)=t_0t_1+F_2(t_2,\ldots,t_n).
	\end{equation}
	Since $\widetilde{F}(\bt)=O(B^2/R^2)$ as prescribed by $w_2$, a reasonable choice for $Q$ is therefore
	\begin{equation}\label{eq:Q}
		Q=\frac{B}{LR},
	\end{equation} with which $$\frac{F(\bx)}{L^2Q^2}=\frac{R^2}{B^2}F(\bx)=O(1).$$
	
	\section{The oscillatory integrals}\label{se:oscint}
 The goal of this section is to manipulate the terms $I_{q,L,\blambda}(w_{B,R};\bc)$ \eqref{eq:Iqc1} and to deduce various estimates for them. We assume $n\geqslant 2$ throughout this section.
 \subsection{Preliminary manipulation}
	 We recall the choice of the matrix $M$ from Lemma \ref{change of variables} and $\widetilde{F}$ from \eqref{eq:tildeF}.
	The change of variables $$\bt=M\bx=M(L\by+\blambda)$$ yields
	\begin{equation}\label{eq:IqL2}
		I_{q,L,\blambda}(w_{B,R};\bc)=\frac{1}{|\det M| L^{n+1}}\e_{qL^2}\left(\bc\cdot \blambda\right) \CI_{q,L}(w_{B,R};\bc),
	\end{equation} where
	\begin{equation}\label{eq:IqL3}
		\CI_{q,L}(w_{B,R};\bc):=\int_{\BR^{n+1}}w_{B,R}(M^{-1}\bt)h\left(\frac{q}{Q},\frac{\widetilde{F}(\bt)}{L^2Q^2}\right)\e_{qL^2}\left(-\widetilde{\bc}\cdot\bt\right)\operatorname{d}\bt,
	\end{equation} $\widetilde{F}$ is defined by \eqref{eq:tildeF}, and, noting that $\bc\cdot\bx= \bc\cdot M^{-1}\bt=(M^{-1})^t \bc\cdot\bt$,
 \begin{equation}\label{eq:bctilde}
		\widetilde{\bc}:=(M^{-1})^t\bc=(\widetilde{c_0},\ldots,\widetilde{c_n}).
	\end{equation}
	is a real vector. 

	Let the new variables $\bs\in\BR^{n+1}$ be defined by $$s_0:=\frac{t_0}{B},\quad s_1:=\frac{R^2}{B}t_1,\quad s_j:=\frac{R}{B}t_j,\quad j\geqslant 2.$$
	Then we have $$\operatorname{d}\bt=\frac{B^{n+1}}{R^{n+1}}\operatorname{d}\bs.$$ Moreover, by our choice of the  matrix $M$,
	$$\frac{\widetilde{F}(\bt)}{L^2Q^2}=\frac{R^2}{B^2}\left(t_0t_1+F_2(t_2,\ldots,t_n)\right)=\widetilde{x_0}\widetilde{x_1}+F_2(\widetilde{x_2}, \ldots, \widetilde{x_n})=\widetilde{F}(\bs).$$

	We also define the real vector $\boldsymbol{d}=\boldsymbol{d}(B,R,L;\cc)$ by \begin{equation}\label{eq:d}
		d_0=\frac{B}{L^2}\widetilde{c_0},\quad d_1=\frac{B}{L^2R^2}\widetilde{c_1},\quad d_j=\frac{B}{L^2R}\widetilde{c_j},\quad 2\leqslant j\leqslant n.
	\end{equation}
	   Hence we can now rewrite
	\begin{equation}\label{eq:cihatci}
		\CI_{q,L}(w_{B,R};\bc)=\frac{B^{n+1}}{R^{n+1}}\widehat{\CI}_{q}(\widehat{w};\boldsymbol{d})
	\end{equation} where
	\begin{equation}\label{eq:Ihat}
		\widehat{\CI}_{q}(\widehat{w};\boldsymbol{d})=\int_{\BR^{n+1}} \widehat{w}(\bs)h\left(\frac{q}{Q},\widetilde{F}(\bs)\right)\e_q\left(-\bd\cdot\bs\right)\operatorname{d}\bs,
	\end{equation} with $$\widehat{w}(\bs):=w_1\left(\frac{s_j}{s_0},2\leqslant j\leqslant n\right)w_2\left(\widetilde{F}(\bs)\right)w_3\left(s_0\right)$$ depending on $w_1,w_2,w_3$.
	We recall that $$\bc\neq\boldsymbol{0}\Leftrightarrow \widetilde{\bc}\neq\boldsymbol{0}\Leftrightarrow \bd\neq\boldsymbol{0}.$$

	\subsection{Simpler and harder estimates}
 Let $\CH$ be the class of functions defined on $\BR_{>0}\times\BR\to\BC$ as in \cite[(7.1)]{H-Bdelta}. For $s\ll 1$, we define $$f_s(y)=sh(s,y),$$ and we note that $f_s(y)\in\CH$. 
	Let $\CF(x_0,\ldots,x_n)$ be a non-degenerate indefinite real quadratic form. For any real vector $\uu\in\BR^{n+1}$, and for an appropriate fixed weight function $w\in\mathcal{C}_0^\infty(\BR^{n+1})$, let
	\begin{equation}\label{eq:J}
		J_s(\CF;\uu)=\int_{\BR^{n+1}}w(\bt)f_s(\CF(\bt))e\left(- \uu\cdot\bt\right)\operatorname{d}\bt.
	\end{equation}
 The following bound follows from \cite[Lemma 15]{H-Bdelta}. For any $\uu \in \mathbb{R}^{n+1}$, we have
	\begin{equation}\label{eq:Jtrivial}
		J_s(\CF;\uu)\ll s.
	\end{equation}
Here, the implied constant may depend on the diameter of the support of $w$, on 
$$\|w\|_0 = \sup_{\bx \in \BR^{m}}|w(\bx)|,$$
and on a lower bound for $\frac{\partial \CF}{\partial x_0}$ on the support of $w$. 
These conditions on the implicit constant in (\ref{eq:Jtrivial}) are related to the introduction of the class of functions $\mathcal{C}_0(S)$ in \cite[$\S$6]{H-Bdelta}.

	The following two theorems are based on Heath-Brown's ``simpler" and ``harder" estimates \cite[Lemma 18, Lemma 22]{H-Bdelta}. However, in order to remove the smooth weights and ultimately prove Theorem \ref{main counting result}, we need to slightly refine Heath-Brown's results so that our estimates have an explicit dependence on $w$. This dependence is realised through the quantity $\|w\|_N$, as defined in (\ref{def partial derivative norm}). We do not reproduce a full proof of these results here, instead only highlighting the refinements to Heath-Brown's arguments. The two results below may also be compared with \cite[Proposition 5.1, Proposition 6.1]{Dymov}.
	
	The first estimate refines \cite[Lemma 18]{H-Bdelta}, which is a consequence of the ``first derivative test'' \cite[Lemma 10]{H-Bdelta} and a Fourier analysis of the function $h(\cdot,\cdot)$ \cite[Lemma 17]{H-Bdelta}:
	
	\begin{lemma}\label{HB simpler estimate}(``simpler estimates")
 For $N\geq 0$ and $\uu\neq 0$ we have
		$$J_s(\CF;\uu)\ll_N (s|\uu|)^{-N}\|w\|_N.$$
	\end{lemma}
	
	\begin{proof}
		Let $w_0:\mathbb{R} \ra \mathbb{R}_{\geq 0}$ be the smooth bump function defined in \cite[Equation (2.1)]{H-Bdelta}. \cite[Lemma 18]{H-Bdelta} is proved by applying integration by parts $N$ times to the integral
		$$ \int_{\mathbb{R}^{n+1}} W(\bt)e(k\mathcal{F}(\bt)-\uu\cdot \bt) \operatorname{d}\bt,$$
		where
		$$ W(\bt) := \frac{w(\bt)}{w_0(\mathcal{F}(\bt)/2K)}$$
		and $K$ is such that $|\mathcal{F}(\bt)|\leq K$ on $\operatorname{supp}(w)$. Since $K=O(1)$, we have $\|W\|_N \ll \|w\|_N$, and hence integrating by parts $N$ times introduces a factor $O(\|w\|_N)$. 
	\end{proof}
	
	The second result is obtained from a more delicate analysis  of the locus where the ``first derivative test'' fails, which involves the Hessian.
	
	\begin{lemma}\label{HB harder estimate}(``harder estimates")
		Suppose that $1 \leq A \leq |\uu|^{1/3}$. Then 
		
		$$J_s(\CF;\uu)\ll_N A^{-N}\|w\|_N + sA^{n+1}|\uu|^{1-\frac{n+1}{2}}.$$
	\end{lemma}
	
	\begin{proof}
		Similarly to the lemma above, Heath-Brown's ``good pairs" are dealt with using integration by parts $N$ times, to give a contribution of $O(A^{-N}\|w\|_N)$. For the ``bad pairs", the integral is estimated trivially, and no derivatives of $w$ are taken. Therefore, this part of the argument goes through unchanged, giving a contribution of $A^{n+1}|\uu|^{1-\frac{n+1}{2}}.$
	\end{proof}

 Lemmas \ref{HB simpler estimate} and \ref{HB harder estimate} furnish alternative ways to estimate the quantity $\widehat{\CI}_{q}(\widehat{w};\boldsymbol{d})$ in \eqref{eq:Ihat}.
	\begin{corollary}\label{co:simplehardest}
		Uniformly for all $q\in\BN,\bd\in\BR^{n+1}$, we have 

		$$\widehat{\CI}_{q}(\widehat{w};\boldsymbol{d})\begin{cases} \ll 1 &\quad \text{``Trivial estimate''};\\
			\ll_N \frac{Q}{q}\left(\frac{|\dd|}{Q}\right)^{-N}\|w\|_N &\quad \text{``Simpler estimate''}.
		\end{cases}$$
  Moreover, for any $1\leq A \leq |\bd /q|^{1/3}$, we have $$ \widehat{\CI}_{q}(\widehat{w};\boldsymbol{d})\ll_{\varepsilon,N} \frac{Q}{q} A^{-N}\|w\|_N + A^{n+1}\left|\frac{\bd}{q}\right|^{1-\frac{n+1}{2}}\quad \text{``Harder estimate''}.$$
	\end{corollary}
	
	\begin{proof}
		On recalling \eqref{eq:Ihat}, we clearly have
		\begin{equation}\label{eq:IJ}
			\begin{split}
				\widehat{\CI}_{q}(\widehat{w};\boldsymbol{d})&=\int_{\BR^{n+1}} \widehat{w}(\bs)h\left(\frac{q}{Q},\widetilde{F}(\bs)\right)\e\left(- \frac{\bd}{q}\cdot\bs\right)\operatorname{d}\bs\\ & \ll \frac{Q}{q}\left|J_{\frac{q}{Q}}\left(\widetilde{F};\frac{\bd}{q}\right)\right|.
			\end{split}
		\end{equation}
		The first estimate is just \eqref{eq:Jtrivial}. Since $F$ (and hence $\widetilde{F}$) is non-degenerate, the Hessian hypothesis is satisfied.  Corollary \ref{co:simplehardest} follows from \eqref{eq:IJ} and Lemmas \ref{HB simpler estimate} and \ref{HB harder estimate} with $$s=\frac{q}{Q},\quad \bu=\frac{\boldsymbol{d}}{q}.\qedhere$$
	\end{proof}
 
	\section{Quadratic exponential sums}\label{se:quadexpsum}
	Analogously to \cite[Lemma 28]{H-Bdelta}, the aim of this section is to the prove the following ``Kloosterman refinement'' result. 
	\begin{theorem}\label{thm:sqsum}
		Assume that $n\geqslant 2$. We have
		$$S_{q,L,\blambda}(\bc)\ll_\varepsilon L^n q^{\frac{n+3}{2}+\varepsilon}\gcd(q,L)^\frac{1}{2},$$
			$$\sum_{q\leqslant X} \left|S_{q,L,\blambda}(\bc)\right|\ll_\varepsilon \begin{cases}
				|\bc|^\varepsilon L^{n+\varepsilon} X^{\frac{n+4}{2}+\varepsilon} &\text{ if } F^*(\bc)\neq 0;\\  L^{n+\varepsilon} X^{\frac{n+5}{2}+\varepsilon}&\text{ if } F^*(\bc)= 0.
			\end{cases}$$
If moreover $n$ is even and $F^*(\bc)=0$, we also have
			$$\sum_{q\leqslant X} \left|S_{q,L,\blambda}(\bc)\right|\ll_\varepsilon L^{n+\varepsilon}X^{\frac{n+4}{2}+\varepsilon}. $$
   All the implied constants above are independent of $q,\bc,L$ and $\blambda$.
	\end{theorem}
	Theorem \ref{thm:sqsum} is concerned with
	estimating the average over $q$ of the absolute value of the finite exponential sums $S_{q,L,\blambda}(\bc)$ from \eqref{eq:Sqc}. As already pointed out in \cite{H-Bdelta}, no ``double Kloosterman refinement'' is needed when $n\geqslant 4$ (while it becomes essential for $n=3$ in \cite{PART2}).
	
	To prove Theorem \ref{thm:sqsum}, we decompose $S_{q,L,\blambda}(\bc)$ into an exponential sum over ``good moduli'' and another over ``bad moduli'', based on the Chinese remainder theorem. For the former we give precise formula, while for the latter we obtain a uniform upper bound. 
	
	\subsection{Decomposing $S_{q,L,\blambda}(\bc)$} For every fixed $q\geqslant 1$, we have a decomposition
 \begin{equation}\label{eq:qdecomp}
		q=q_1q_2,
	\end{equation} where $$\gcd(q_1,2L\Delta_F)=1,\quad q_2\mid (2L\Delta_F)^\infty.$$
	Since $\gcd(q_1, q_2) =1$, by the Chinese remainder theorem, every $a \bmod q$ can be written as $$a=a_1q_2+a_2q_1,$$ for a unique choice of $a_1\bmod q_1,a_2\bmod q_2$. Similarly, for every $\bsigma\in (\BZ/qL\BZ)^{n+1}$ we can write $$\bsigma=q_1\bsigma_2+q_2L \bsigma_1$$ for unique $\bsigma_1\bmod q_1$ and $\bsigma_2\bmod q_2L$. If we let $$k:=\frac{F(\blambda)}{L},$$ then we also write $$k=k_2q_1+k_1q_2L,$$ with $k_1\bmod q_1$ and $k_2\bmod q_2L$.
	With these decompositions we obtain the following lemma. 
	\begin{lemma}\label{le:sqdecomp}
	We have	$$S_{q,L,\blambda}(\bc)=S^{(1)}_{q,L,\blambda}(\bc)S^{(2)}_{q,L,\blambda}(\bc),$$
	where \begin{equation}\label{eq:S1}
		S^{(1)}_{q,L,\blambda}(\bc):=\sum_{\bsigma_1\in (\BZ/q_1\BZ)^{n+1}}\sum_{\substack{a_1\bmod q_1\\(a_1,q_1)=1}}\e_{q_1}\left(a_1\left((q_2L)^2F(\bsigma_1)+q_2(\nabla F(\blambda)\cdot \bsigma_1+k_1)\right)+\bc\cdot\bsigma_1\right),
	\end{equation} \begin{equation}\label{eq:S2}
		S^{(2)}_{q,L,\blambda}(\bc):=\sum_{\substack{\bsigma_2\in (\BZ/q_2L\BZ)^{n+1}\\ H_{\blambda,L}(q_1\bsigma_2)\equiv 0\bmod L}}\sum_{\substack{a_2\bmod q_2\\(a_2,q_2)=1}}\e_{q_2L}\left(a_2\left(q_1^2L F(\bsigma_2)+q_1(\nabla F(\blambda)\cdot \bsigma_2+k_2)\right)+\bc\cdot\bsigma_2\right).
	\end{equation} 	\end{lemma}
\begin{proof}
 We observe that the condition $L\mid H_{\boldsymbol{\lambda},L}(q_1\bsigma_2)$, namely $$k_2q_1+q_1\nabla F(\blambda)\cdot\bsigma_2\equiv 0\bmod L,$$ is equivalent to $$k_2q_1+q_1\nabla F(\blambda)\cdot\bsigma_2\equiv 0\bmod q_1 L,$$  because  $\gcd(q_1,L)=1$.  The decomposition can then be obtained by direct computation.
\end{proof}

 \subsection{Evaluating $S^{(1)}_{q,L,\blambda}(\bc)$}	We start by proving a general version of Kloosterman--Salié-type sums, which is a direct analogue of \cite[Lemma 26]{H-Bdelta}. 
 
 Let $G$ be an integral quadratic form in $(n+1)$-variables with matrix $\CM_G$, and let $\Delta_G:=\det\CM_G$. That is, $\CM_G$ is a symmetric $(n+1)\times(n+1)$-matrix such that $2\CM_G$ has integer entries and $G(\bx)=\bx^t \CM_G\cdot\bx$.

 Let $G^*$ be the dual quadratic form associated to $G$, which is defined by the adjugate $\CM_G^{\operatorname{adj}}:=(\Delta_G) \CM_G^{-1}$ of $\CM_G$. Note that $2^{n+1}\Delta_G\in\BZ$, and $2^n\CM_G^{\operatorname{adj}}$ has integer entries.
	\begin{lemma}\label{le:Tq}
	  Let $G$ be as above and let $m\in\BZ,\bc\in\BZ^{n+1}$ and $q\in\BN$. Let us define
		$$T_q(G,m;\bc):=\sum_{\substack{a\bmod q\\(a,q)=1}}\sum_{\bb\in(\BZ/q\BZ)^{n+1}}\e_q\left(a(G(\bb)-m)+\bc\cdot\bb\right).$$ Assume that $G$ is non-degenerate and $\gcd(q,2\Delta_G)=1$.
		Then we have
		$$T_q(G,m;\bc)=\left(\iota_q\left(\frac{2}{q}\right)\right)^{n+1}\left(\frac{2^{n+1}\Delta_G}{q}\right) \sum_{\substack{a\bmod q\\(a,q)=1}}\left(\frac{a}{q}\right)^{n+1}\e_q\left(-am-\overline{a2^{n+2}\Delta_G}2^nG^*(\bc)\right),$$ where for any $d\bmod q$ with $(d,q)=1$, we write $\overline{d}\bmod q$ for the multiplicative inverse, and $$\iota_q:=\sum_{x\bmod q}\e_q(x^2)$$ is the quadratic Gauss sum of modulus $q$.
	\end{lemma}
 Let us deduce
 \begin{proposition}\label{prop:S1}
		We have
		$$S^{(1)}_{q,L,\blambda}(\bc)=\iota_{q_1}^{n+1}\left(\frac{2^{n+1}\Delta_F^n}{q_1}\right)\e_{q_1}\left(-\overline{q_2L^2}\blambda\cdot \bc\right)\sum_{\substack{a\bmod q_1\\(a,q_1)=1}}\left(\frac{a}{q_1}\right)^{n+1}\e_{q_1}\left(-a2^nF^*(\bc)\right).$$
	\end{proposition}
	\begin{proof}[Proof of Proposition \ref{prop:S1}]
		The change of variables for the $a_1$-sum $a_1\mapsto a_1':=a_1\overline{L^2}$ gives
		$$S^{(1)}_{q,L,\blambda}(\bc)=\sum_{\bsigma_1\in (\BZ/q_1\BZ)^{n+1}}\sum_{\substack{a_1'\bmod q_1\\(a_1',q_1)=1}}\e_{q_1}\left(a_1' F(q_2L^2\bsigma_1+\blambda)+\bc\cdot\bsigma_1\right).$$
		The second non-singular change of variables for the $\bsigma_1$-sum $\bsigma_1\mapsto \bb:=q_2L^2\bsigma_1+\blambda$ gives, with the notation of Lemma \ref{le:Tq},
		$$S^{(1)}_{q,L,\blambda}(\bc)=\e_{q_1}\left(-\overline{q_2L^2}\bc\cdot\blambda\right)T_{q_1}(F,0;\overline{q_2L^2}\bc).$$ Applying Lemma \ref{le:Tq} and the final change of variables $a\mapsto \overline{a2^{n+2}\Delta_F(q_2L^2)^2}$ directly yields the desired formula.
	\end{proof}
	\begin{corollary}\label{co:s1qbound}
		Uniformly for all $q,\bc,L,\blambda$, we have
		$$S^{(1)}_{q,L,\blambda}(\bc)\ll q_1^{\frac{n+2}{2}}\tau(q_1)\gcd(q_1,2^nF^*(\bc))^\frac{1}{2}.$$
		When $n$ is even and $F^*(\bc)=0$, we have
		$$\left|S^{(1)}_{q,L,\blambda}(\bc)\right|=\begin{cases}
			q_1^{\frac{n+1}{2}}\phi(q_1) &\text{ if } q_1=\square;\\ 0 &\text{ otherwise}.
		\end{cases}$$
	\end{corollary}
	\begin{proof}[Proof of Corollary \ref{co:s1qbound}]
	Since $2\nmid q_1$, using \cite[(3.38)]{I-K}, we have $|\iota_{q_1}|=q_1^\frac{1}{2}$. Then clearly $$|S^{(1)}_{q,L,\blambda}(\bc)|=q_1^{\frac{n+1}{2}} \left|\sum_{\substack{a\bmod q_1\\(a,q_1)=1}}\left(\frac{a}{q_1}\right)^{n+1}\e_{q_1}\left(-a2^nF^*(\bc)\right)\right|.$$
 When $n$ is odd (resp. even), the inner exponential sum is a Ramanujan sum (resp. Gauss sum), which can both be bounded by $$\ll \tau(q_1)q_1^\frac{1}{2}\gcd(q_1,2^nF^*(\bc))^\frac{1}{2}$$ by multiplicativity and by Weil's bound (see \cite[Corollary 11.12]{I-K}).
		
		Now assume that $n$ is even and $F^*(\bc)=0$. Then the inner sum is a character sum (recall that $2\nmid q_1$):
		$$\sum_{a\bmod q_1}\left(\frac{a}{q_1}\right)=\begin{cases}
		0 & \text{ if } q_1\neq\square;\\ \phi(q_1) &\text{ otherwise}.
		\end{cases}$$ (See also \cite[Proof of Lemma 26]{H-Bdelta}.)
	\end{proof}

	\begin{proof}[Proof of Lemma \ref{le:Tq}]
	We write $\CM:=2\CM_G$ and we fix a matrix $R\in\operatorname{GL}_{n+1}(\BZ/q\BZ)$ such that $$R^t \CM R=\operatorname{diag}(\beta_1,\ldots,\beta_{n+1}),$$ with $\gcd(\beta_i,q)=1$. Then the change of variables $\bb\mapsto \bx:=R^{-1}\bb$ yields 
		\begin{align*}
			T_q(G,m;\bc)&=\sum_{\substack{a\bmod q\\(a,q)=1}}\e_q(-am)\sum_{\bx\in(\BZ/q\BZ)^{n+1}}\e_q\left(\overline{2}a\sum_{i=1}^{n+1}\beta_i x_i^2+\bc\cdot R\bx\right)\\ &=\sum_{\substack{a\bmod q\\(a,q)=1}}\e_q(-am)\prod_{i=1}^{n+1}\sum_{x\bmod q}\e_q\left(\overline{2}a\beta_i x^2+d_{i}x\right),
		\end{align*}
		where we have written $$\bd=(d_1,\ldots,d_{n+1}):=R^t\bc.$$
		For every fixed $i$ and $a\bmod q$ with $(a,q)=1$, \begin{align*}
			&\sum_{x\bmod q}\e_q\left(\overline{2}a\beta_i x^2+d_i x\right)\\
			=& \e_q\left(-\overline{2a\beta_i}d_i^2\right)\sum_{x\bmod q}\e_q\left(\overline{2}a\beta_i\left(x+\overline{a\beta_i}d_i\right)^2\right)\\ =& \e_q\left(-\overline{2 a\beta_i}d_i^2\right)\left(\frac{\overline{2}a\beta_i}{q}\right)\iota_q,
		\end{align*} where for the last equality we have executed the change of variables $x\mapsto x+\overline{a\beta_i}d_i$ and made use of \cite[(3.38)]{I-K}.
		Notice that modulo $q$ we have \begin{align*}
			\sum_{i=1}^{n+1}\overline{\beta_i}d_i^2&\equiv \bc^t R\operatorname{diag}(\overline{\beta_1},\ldots,\overline{\beta_{n+1}})\cdot R^t\bc\\ &\equiv \overline{\det \CM}\bc^t RR^{-1}\CM^{\operatorname{adj}} (R^{-1})^t\cdot R^t\bc\\&\equiv \overline{\det \CM}\bc^t \CM^{\operatorname{adj}}\cdot \bc\\            &=\overline{2^{n+1}\Delta_G}2^n G^*(\bc).
		\end{align*}
		Hence 
		\begin{align*}
\prod_{i=1}^{n+1}\e_q\left(-\overline{2 a\beta_i}d_i^2\right)\left(\frac{\overline{2}a\beta_i}{q}\right)&=\e_q\left(-\overline{2a}\sum_{i=1}^{n+1}\overline{\beta_i}d_i^2 \right)\left(\frac{\overline{2}a}{q}\right)^{n+1}\left(\frac{\prod_{i=1}^{n+1}\beta_i}{q}\right)\\ &=\e_q\left(-\overline{a2^{n+2}\Delta_G}2^nG^*(\bc)\right)\left(\frac{2a}{q}\right)^{n+1}\left(\frac{2^{n+1}\Delta_G}{q}\right).
		\end{align*}  We therefore conclude the desired formula.
	\end{proof}

\subsection{Estimating $S^{(2)}_{q, L, \blambda}(\bc)$}
	\begin{proposition}\label{prop:S2}
		Assume that $\blambda=(\lambda_0,\cdots,\lambda_n)\in \mathbb{Z}^{n+1}$ such that $\gcd(\lambda_0,\ldots, \lambda_n,L)=1$. Then, uniformly for all $q,\bc,L,\blambda$, we have
		$$S^{(2)}_{q,L,\blambda}(\bc)\ll q_2^{\frac{n+3}{2}}L^n\gcd(q_2,L)^\frac{1}{2}.$$
	\end{proposition}
	Proposition \ref{prop:S2} is a direct analogue of \cite[Lemma 25]{H-Bdelta}, from which our proof is also inspired. The main challenge here is that we need to not only get a best ``square-root cancellation'' at the level $q_2$, but also to optimize the dependency on $L$.
	\begin{proof}
		By the Cauchy--Schwarz inequality, we have that $$\left|S^{(2)}_{q,L,\blambda}(\bc)\right|^2\leqslant\phi(q_2)\sum_{\substack{a_2\bmod q_2\\(a_2,q_2)=1}} \CS_{q,L}(a_2,\bc),$$ where for each fixed $a_2$, we let 
		\begin{align*}
			& \CS_{q,L}(a_2,\bc):=\left|\sum_{\substack{\bsigma_2\in (\BZ/q_2L\BZ)^{n+1}\\ H_{\blambda,L}(q_1\bsigma_2)\equiv 0\bmod L}}\e_{q_2L}\left(a_2\left(q_1^2L F(\bsigma_2)+q_1(\nabla F(\blambda)\cdot \bsigma_2+k_2)\right)+\bc\cdot\bsigma_2\right)\right|^2 \\ = &\sum_{\substack{\bnu_1,\bnu_2\bmod q_2L\\ H_{\blambda,L}(q_1\bnu_i)\equiv 0\bmod L}}\e_{q_2L} \left(a_2\left(q_1^2L(F(\bnu_1)-F(\bnu_2))+q_1\nabla F(\blambda)\cdot(\bnu_1-\bnu_2)\right)+\bc\cdot(\bnu_1-\bnu_2)\right)\\ = & \sum_{\substack{\bnu_0:\nabla F(\blambda)\cdot\bnu_0 \equiv0\bmod L\\ \bnu_1:H_{\blambda,L}(q_1\bnu_1)\equiv 0\bmod L}}\e_{q_2L} \left(a_2\left(q_1^2L(\nabla F(\bnu_0)\cdot \bnu_1-F(\bnu_0))+q_1\nabla F(\blambda)\cdot \bnu_0\right)+\bc\cdot\bnu_0\right),
		\end{align*} where for the last equality we have set $\bnu_0=\bnu_1-\bnu_2$.
		
		We proceed by fixing $a_2,\bnu_0$ and first sum over $\bnu_1$. For this purpose we rewrite 
		$$\CS_{q,L}(a_2,\bc)=\sum_{\substack{\bnu_0\bmod q_2L\\ \nabla F(\blambda)\cdot\bnu_0 \equiv0\bmod L}}\e_{q_2L} \left(a_2q_1\left(-q_1 L F(\bnu_0)+\nabla F(\blambda)\cdot \bnu_0\right)+\bc\cdot\bnu_0\right) T_{q,L}(a_2,\bnu_0),$$
		where for each fixed $\bnu_0$, we write $$T_{q,L}(a_2,\bnu_0):=\sum_{\bnu_1:H_{\blambda,L}(q_1\bnu_1)\equiv 0\bmod L}\e_{q_2}\left(a_2q_1^2\nabla F(\bnu_0)\cdot \bnu_1\right).$$
		The condition $H_{\blambda,L}(q_1\bnu_1)\equiv 0\mod L$ can be resolved by adding an additive character sum of modulus $L$, so that the $\bnu_1$-variable ranges over all vectors modulo $q_2L$:
		\begin{align*}
			T_{q,L}(a_2,\bnu_0)&=\sum_{\bnu_1\bmod q_2L}\frac{1}{L}\sum_{l\bmod L}\e_L\left(l H_{\blambda,L}(q_1\bnu_1)\right)\e_{q_2}\left(a_2q_1^2\nabla F(\bnu_0)\cdot \bnu_1\right)\\ &=\frac{1}{L}\sum_{l\bmod L}\e_L(lk)\CT_{q,L}(a_2,l,\bnu_0),
		\end{align*} where for every fixed $l$, we let
		\begin{equation*}
			\begin{split}
				\CT_{q,L}(a_2,l,\bnu_0)&:=\sum_{\bnu_1\bmod q_2L}\e_{q_2L}\left(\left(a_2q_1^2L\nabla F(\bnu_0)+q_1q_2l\nabla F(\blambda)\right)\cdot \bnu_1\right)\\ &=\begin{cases}
					(q_2L)^{n+1} &\text{ if } q_2L\mid a_2q_1^2L\nabla F(\bnu_0)+q_1q_2l\nabla F(\blambda);\\ 0&\text{ otherwise}.
				\end{cases}
			\end{split}
		\end{equation*}
		
		We are now reduced to counting the number of $\bnu_0$ and $l$ such that $\CT_{q,L}(a_2,l,\bnu_0)\neq 0$. Write \begin{equation}\label{eq:decompnu0}
			\bnu_0=\lcm(q_2,L)\balpha_1+\balpha_2,
		\end{equation} with $\balpha_1\bmod \gcd(q_2,L)$ and $\balpha_2\bmod \lcm(q_2,L)$. Then the number of $\bnu_0$ is entirely determined by the number of $\balpha_1$ and that of $\balpha_2$. Note that the divisibility condition above is equivalent to \begin{equation}\label{eq:conddivl}
			\frac{q_2L}{\gcd(q_2,L)}\mid a_2q_1^2\frac{L}{\gcd(q_2,L)}\nabla F(\balpha_2)+\frac{q_2}{\gcd(q_2,L)}q_1l\nabla F(\blambda).
		\end{equation}
		So $\CT_{q,L}(a_2,l,\bnu_0)$ only depends on $\balpha_2$ and we may write $$\CT_{q,L}(a_2,l,\bnu_0)=\CT_{q,L}(a_2,l,\balpha_2).$$
		On the other hand, the condition $\nabla F(\blambda)\cdot\bnu_0 \equiv0\bmod L$ is equivalent to \begin{equation}\label{eq:cond2prime} 
			\nabla F(\blambda)\cdot\balpha_2 \equiv0\bmod L.
		\end{equation}
		The conditions \eqref{eq:conddivl} and \eqref{eq:cond2prime} are only concerned with $l$ and $\balpha_2$. 
		
		With the decomposition \eqref{eq:decompnu0}, we have, modulo $q_2L$,
		\begin{multline*}
			a_2q_1\left(-q_1 L F(\bnu_0)+\nabla F(\blambda)\cdot \bnu_0\right)+\bc\cdot\bnu_0\\ \equiv -a_2q_1^2 L F(\balpha_2)+(a_2q_1\nabla F(\blambda)+\bc)\cdot \balpha_2+ (a_2q_1\nabla F(\blambda)+\bc)\cdot\lcm(q_2,L)\balpha_1.
		\end{multline*}
	Concerning $\balpha_1$, if $a_2$ satisfies the condition \begin{equation}\label{eq:conda2}
		\gcd(q_2,L)\mid a_2q_1\nabla F(\blambda)+\bc,
	\end{equation} then \begin{align*}
		&\sum_{\balpha_1\bmod \gcd(q_2,L)}\e_{q_2L}\left((a_2q_1\nabla F(\blambda)+\bc)\cdot\lcm(q_2,L)\balpha_1\right)\\ =&\sum_{\balpha_1\bmod \gcd(q_2,L)}\e_{\gcd(q_2,L)}\left((a_2q_1\nabla F(\blambda)+\bc)\cdot\balpha_1\right)=
			\gcd(q_2,L)^{n+1},
	\end{align*} and we can write  
		$$\CS_{q,L}(a_2,\bc)=\frac{1}{L}(q_2L)^{n+1}\gcd(q_2,L)^{n+1}\sum_{l\bmod L}\e_L(lk)U_{q,L}(a_2,l)$$ where for every fixed $a_2$ and $l$, \begin{align*}
			U_{q,L}(a_2,l)&:=\sum_{\substack{\balpha_2\bmod \lcm(q_2,L)\\ \eqref{eq:conddivl}\text{ and } \eqref{eq:cond2prime} \text{ hold}}}\e_{q_2L}\left(-a_2q_1^2 L F(\balpha_2)+(a_2q_1\nabla F(\blambda)+\bc)\cdot \balpha_2\right).
		\end{align*} Otherwise $\CS_{q,L}(a_2,\bc)=0$ if \eqref{eq:conda2} is not satisfied.

	We recall that $\CM_F$ denotes the matrix representation of $F$, $\CM_F^{\operatorname{adj}}$ denotes its adjugate. 
	Then if  $a\in\BN$ is such that $a\mid L^\infty$ and $a\mid \nabla F(\blambda)=2\CM_F\blambda$, on multiplying with $2^n\CM_F^{\operatorname{adj}}$, this implies that $a\mid 2^{n+1}\Delta_F \blambda$, whence $a=O(1)$, where the implicit constant only depends on the discriminant of $F$. 
	
		The condition \eqref{eq:conddivl} implies that $l$ satisfies  
	$\frac{L}{\gcd(q_2,L)}\mid \frac{q_1q_2}{\gcd(q_2,L)}l\nabla F(\blambda)$, itself equivalent to \begin{equation}\label{eq:condl}
		\frac{L}{\gcd(q_2,L)}\mid l\nabla F(\blambda),
	\end{equation} since $\gcd(\frac{L}{\gcd(q_2,L)}, \frac{q_1q_2}{\gcd(q_2,L)})=1$.
	So, on writing $$l=\frac{L}{\gcd(q_2,L)}l'+l'',$$ with $l'\bmod \gcd(q_2,L)$ and $l''\bmod\frac{L}{\gcd(q_2,L)}$, then \eqref{eq:condl} is equivalent to $	\frac{L}{\gcd(q_2,L)}\mid l''\nabla F(\blambda)$ and hence determines $l''$ up to $O(1)$.
	Hence $$\#\{l\bmod L:\eqref{eq:condl} \text{ holds}\}=O(\gcd(q_2,L)).$$

	We next estimate trivially $U_{q,L}(a_2,l)$ by counting the number of $\balpha_2$ satisfying \eqref{eq:conddivl} and \eqref{eq:cond2prime}.
Note that \eqref{eq:conddivl} is equivalent to $$ \nabla F(\balpha_2)\equiv -\overline{a_2q_1}\frac{q_2}{\gcd(q_2,L)}\left(\frac{L}{\gcd(q_2,L)}\right)^{-1}l\nabla F(\blambda)\bmod q_2.$$ On multiplying with $2^n\CM_F^{\operatorname{adj}}$, this implies that, once $a_2$ and $l$ are chosen, there are $O(1)$ choices for the residues of the vector $\balpha_2$  modulo $q_2$  satisfying \eqref{eq:conddivl}.  Again since the coordinates of $\blambda$
have a greatest common divisor which is coprime to $L$, we can do an invertible linear change of variables 
$\balpha_2\mapsto \balpha_2'\mod L$ 
so that the first coordinate say $\alpha_{2,1}'$ of  $\balpha_2^\prime \bmod L$ is  $\balpha_2\cdot\blambda^\prime$ where $\blambda'\in\BZ^{n+1}$ also has coprime coordinates away from $L$ and is defined by $\nabla F(\blambda)=b\blambda'$ with certain $b=O(1)$. Then the condition \eqref{eq:cond2prime} determines $\alpha_{2,1}'\bmod L$ up to $O(1)$, and if moreover $l$ is also fixed, besides the $O(1)$ choices for $\balpha_2'\bmod q_2$, in addition there are only $O(1)$ choices for one of coordinates of $\balpha_2'\bmod \lcm(q_2,L)$ by the Chinese remainder theorem. To summarise, when $a_2$ and $l$ are fixed and satisfy \eqref{eq:conda2} and \eqref{eq:condl} respectively,  $$\#\{\balpha_2\bmod \lcm(q_2,L): \eqref{eq:conddivl} \text{ and } \eqref{eq:cond2prime} \text{ hold}\}\ll\left(\frac{\lcm(q_2,L)}{q_2}\right)^n= \left(\frac{L}{\gcd(q_2,L)}\right)^n.$$
		Therefore 
		\begin{align*}
				\CS_{q,L}(a_2,\bc)&\leqslant\frac{1}{L}(q_2L)^{n+1}\gcd(q_2,L)^{n+1}\sum_{\substack{l\bmod L\\ \eqref{eq:condl} \text{ holds}}}\sum_{\substack{\balpha_2\bmod \lcm(q_2,L)\\\eqref{eq:conddivl}\text{ and } \eqref{eq:cond2prime} \text{ hold}}}1\\ &\ll q_2^{n+1}L^{2n}\gcd(q_2,L)^2.
		\end{align*}

		We decompose $$a_2=\gcd(q_2,L)a_2'+a_2''$$ with $a_2'\bmod \frac{q_2}{\gcd(q_2,L)}$ and $a_2''\bmod \gcd(q_2,L)$. Then the condition \eqref{eq:conda2} determines $a_2''$ up to $O(1)$ by the same argument applied to $l$ before. Hence
		$$\#\{a_2\bmod q_2: \eqref{eq:conda2} \text{ holds}\}=O\left(\frac{q_2}{\gcd(q_2,L)}\right).$$ We finally get 	$$\left|S^{(2)}_{q,L,\blambda}(\bc)\right|^2\ll\phi(q_2)\sum_{\substack{a_2\bmod q_2,(a_2,q_2)=1\\ \eqref{eq:conda2}\text{ holds}}} \CS_{q,L}(a_2,\bc)\ll q_2^{n+3}L^{2n}\gcd(q_2,L),$$ 
	 whence the desired bound for $S^{(2)}_{q,L,\blambda}(\bc)$.
	\end{proof}
	\subsection{Proof of Theorem \ref{thm:sqsum}}
		The uniform estimate follows from Corollary \ref{co:s1qbound} and Proposition \ref{prop:S2}, on recalling \eqref{eq:qdecomp} and Lemma \eqref{le:sqdecomp}.
		
		Now assume that $F^*(\bc)\neq 0$. Then by Corollary \ref{co:s1qbound} and Proposition \ref{prop:S2}, we have
		\begin{align*}
			S_{q,L,\blambda}(\bc)\ll_\varepsilon L^n q^{\frac{n+2}{2}+\varepsilon} \gcd(q_1,2^n F^*(\bc))^\frac{1}{2}q_2^\frac{1}{2}\gcd(q_2,L)^\frac{1}{2},
		\end{align*} where we recall the decomposition \eqref{eq:qdecomp}.
		We therefore have
		\begin{align*}
			\sum_{q\leqslant X}\left|S_{q,L,\blambda}(\bc)\right|&\ll_\varepsilon L^n \sum_{T\nearrow X}T^{\frac{n+2}{2}+\varepsilon}\sum_{\substack{q_2\leqslant T\\ q_2\mid (2L\Delta_F)^\infty}} q_2^\frac{1}{2}\gcd(q_2,L)^\frac{1}{2}\sum_{\substack{q_1\sim  \frac{T}{q_2}\\ (q_1,2L\Delta_F)=1}}\gcd(q_1,2^n F^*(\bc))^\frac{1}{2}\\ &\ll_\varepsilon L^n \sum_{T\nearrow X}T^{\frac{n+2}{2}+\varepsilon}\sum_{\substack{q_2\leqslant T\\ q_2\mid (2L\Delta_F)^\infty}}q_2^\frac{1}{2}\gcd(q_2,L)^\frac{1}{2} \times\tau(2^n F^*(\bc)) \frac{T^{1+\varepsilon}}{q_2}\\ &\ll_\varepsilon L^{n }\sum_{T\nearrow X} T^{\frac{n+4}{2}+\varepsilon} |2^n F^*(\bc)|^\varepsilon \#\{q_2\leqslant T:q_2\mid (2L\Delta_F)^\infty\}\\ &\ll_\varepsilon |\bc|^\varepsilon L^{n+\varepsilon} \sum_{T\nearrow X}T^{\frac{n+4}{2}+\varepsilon}\ll_\varepsilon |\bc|^\varepsilon L^{n+\varepsilon} X^{\frac{n+4}{2}+\varepsilon}.
		\end{align*}

  Now if $F^*(\bc)=0$, then $$S_{q,L,\blambda}(\bc)\ll_\varepsilon L^n q^{\frac{n+3}{2}+\varepsilon} \gcd(q_2,L)^\frac{1}{2},$$ and hence
		\begin{align*}
			\sum_{q\leqslant X}\left|S_{q,L,\blambda}(\bc)\right|&\ll_\varepsilon  L^n \sum_{T\nearrow X}T^{\frac{n+3}{2}+\varepsilon}\sum_{\substack{q_2\leqslant T\\ q_2\mid (2L\Delta_F)^\infty}} \gcd(q_2,L)^\frac{1}{2}\sum_{\substack{q_1\sim  \frac{T}{q_2}\\ (q_1,2L\Delta_F)=1}}1 \\ &\ll_\varepsilon   L^n \sum_{T\nearrow X}T^{\frac{n+5}{2}+\varepsilon}\sum_{\substack{q_2\leqslant T\\ q_2\mid (2L\Delta_F)^\infty}}1\ll_\varepsilon L^{n+\varepsilon} X^{\frac{n+5}{2}+\varepsilon}
		\end{align*}
		
		Now assume that $n$ is even and $F^*(\bc)=0$. Using the second half of Corollary \ref{co:s1qbound} we have
		\begin{align*}
				\sum_{q\leqslant X}\left|S_{q,L,\blambda}(\bc)\right|&\ll_\varepsilon  L^n \sum_{T\nearrow X}\sum_{\substack{q_2\leqslant T\\ q_2\mid (2L\Delta_F)^\infty}} q_2^{\frac{n+3}{2}}\gcd(q_2,L)^\frac{1}{2}\sum_{\substack{q_1\sim  \frac{T}{q_2},q_1=\square\\ (q_1,2L\Delta_F)=1}}q_1^{\frac{n+3}{2}}\\ &\ll_{\varepsilon}L^n \sum_{T\nearrow X} T^{\frac{n+4}{2}}\#\{q_2\leqslant T: q_2\mid (2L\Delta_F)^\infty\}\\ &\ll_{\varepsilon}L^{n+\varepsilon} X^{\frac{n+4}{2}+\varepsilon}.
		\end{align*}
		The proof is thus finished.
	\qed
 
	\section{The contribution from $\bc\neq\boldsymbol{0}$}\label{se:cneq0}
	\begin{theorem}\label{thm:cnot0}
	Let $A>0$ be a real parameter. Assume $n\geqslant 4$, $\eta^{-1}\ll B^{A}$ and assume that $\gcd(\lambda_0,\ldots, \lambda_n,L)=1$. Suppose that \eqref{eq:tau} holds with $-1<\tau<1$. We have 
		$$\sum_{q=1}^{\infty}\sum_{\bc\in\BZ^{n+1}\setminus\{\boldsymbol{0}\}}\frac{S_{q,L,\blambda}(\bc)I_{q,L,\blambda}(w_{B,R}^{\pm};\bc)}{(qL)^{n+1}}\ll_{\varepsilon,A} \left(\frac{B}{LR}\right)^{n+1}\times \begin{cases}
			L^{-1}\eta^{\frac{9-5n}{2}}B^{-\frac{n-3}{2}\tau+\varepsilon} &\text{ if } n\geqslant 5;\\
			L^{-1}\eta^{-15/2}B^{-\tau+\varepsilon} & \text{ if } n=4.
		\end{cases}$$
  The implied constant is independent of $R,L,\blambda, \eta$.
	\end{theorem}
	
	\subsection{A lattice condition on $\bc$.}
	The goal is to prove the following constraint on the summand $\bc$ (see also \cite[Lemma 4.2]{Sardari}).
	\begin{proposition}\label{prop:condc}
		We have $S_{q,L,\blambda}(\bc)\neq 0$ only if $\bc$ belongs to the lattice \begin{equation}\label{eq:condc}
			\Gamma_{\blambda,L}:=\{\bd\in\BZ^{n+1}:\text{there exists } b\in\BZ\text{ such that } \bd\equiv b\nabla F(\blambda)\bmod L\}.
		\end{equation}
	\end{proposition}
	\begin{proof}
		We recall the expression \eqref{eq:Sqc}. On resolving the condition $H_{\blambda,L}(\bsigma)\equiv 0\bmod L$, we can rewrite
		$$S_{q,L,\blambda}(\bc)=\sum_{\bsigma\bmod qL}\frac{1}{L}\sum_{l\bmod L}\sum_{\substack{a\bmod q\\(a,q)=1}}\CS_{q,L,\blambda}(a,l;\bc,\bsigma),$$ where \begin{align*}
			\CS_{q,L,\blambda}(a,l;\bc,\bsigma)&:=\e_{qL}\left(a\left(H_{\boldsymbol{\lambda},L}(\boldsymbol{\sigma})+LF(\boldsymbol{\sigma})\right)+\bc\cdot\bsigma\right)\e_L\left(lH_{\blambda,L}(\bsigma)\right)\\ &=\e_q(aF(\bsigma))\e_{qL}\left(\left((a+ql)\nabla F(\blambda)+\bc\right)\cdot\bsigma+(a+ql)k\right).
		\end{align*}
		On substituting $b=a+ql$, we can rewrite 
		$$S_{q,L,\blambda}(\bc)=\frac{1}{L}\sum_{\substack{b\bmod qL\\(b,q)=1}}\e_{qL}(bk)\CS^{(0)}_{q,L,\blambda}(b;\bc),$$
		where $$\CS^{(0)}_{q,L,\blambda}(b;\bc):=\sum_{\bsigma\bmod qL}\e_q(b F(\bsigma))\e_{qL}\left((b\nabla F(\blambda)+\bc)\cdot\bsigma\right).$$
		So $S_{q,L,\blambda}(\bc)\neq 0$ only if there exists $b\bmod qL$ such that $\CS^{(0)}_{q,L,\blambda}(b;\bc)\neq 0$. 
		Observe that $$\CS^{(0)}_{q,L,\blambda}(b;\bc)=\sum_{\bw\bmod q}\e_q\left(bF(\bw)\right)\CS^{(1)}_{q,L,\blambda}(b;\bc,\bw),$$
		where
		$$\CS^{(1)}_{q,L,\blambda}(b;\bc,\bw):=\sum_{\substack{\bsigma\bmod qL\\\bsigma\equiv\bw\bmod q}}\e_{qL}\left(\left(b\nabla F(\blambda)+\bc\right)\cdot\bsigma\right).$$
		So $\CS^{(0)}_{q,L,\blambda}(b;\bc)\neq 0$ only if there exists $\bw\bmod q$ such that $\CS^{(1)}_{q,L,\blambda}(b;\bc,\bw)\neq 0$. On resolving the condition $\bsigma\equiv\bw\bmod q$, we have
		\begin{align*}
			\CS^{(1)}_{q,L,\blambda}(b;\bc,\bw)&=\sum_{\substack{\bsigma\bmod qL}}\frac{1}{q^{n+1}}\sum_{\be\bmod q}\e_q\left((\bsigma-\bw)\cdot\be\right)\e_{qL}\left(\left(b\nabla F(\blambda)+\bc\right)\cdot\bsigma\right)\\ &=\frac{1}{q^{n+1}}\sum_{\be\bmod q}\e_q\left(-\bw\cdot\be\right)	\CS^{(2)}_{q,L,\blambda}(b;\bc,\bw,\be),
		\end{align*} where $$\CS^{(2)}_{q,L,\blambda}(b;\bc,\bw,\be):=\sum_{\substack{\bsigma\bmod qL}}\e_{qL}\left(\left(b\nabla F(\blambda)+\bc+L\be\right)\cdot\bsigma\right).$$
		Now it is clear that $\CS^{(1)}_{q,L,\blambda}(b;\bc,\bw)\neq 0$ only if there exists $\be\bmod q$ such that $\CS^{(2)}_{q,L,\blambda}(b;\bc,\bw,\be)\neq 0$, which is the case if and only if $$qL\mid b\nabla F(\blambda)+\bc+L\be.$$ In particular, a necessary condition for this is $$\bc\equiv -b\nabla F(\blambda)\bmod L.$$
		The proof is thus completed.
	\end{proof}

	\subsection{Proof of Theorem \ref{thm:cnot0}}
	Recalling the choice of $Q$ from \eqref{eq:Q}, we have
 $$Q \gg B^{\frac{1}{2}(1+\tau)}\gg 1.$$
	Moreover, recalling the definition of $\widetilde{\bc}$ \eqref{eq:bctilde} and $\bd$ \eqref{eq:d}, we have \begin{equation}\label{eq:ddcc}
		|\dd|\geqslant \frac{Q}{LR}\max_{0\leqslant i\leqslant n}|\widetilde{c_i}|\gg B^\tau |\bc|.
	\end{equation}  We shall use this inequality repeatedly. We recall also the assumption that $\|w\|_N \ll_N \eta^{-N}$. 
	
	In view of Proposition \ref{prop:condc}, we consider
	$$\CG(L,R;B):= \underset{\substack{\bc\in\BZ^{n+1}\setminus\{\boldsymbol{0}\}:\bc\in\Gamma_{\blambda,L}\\q\ll Q}}{\sum\sum}\frac{\left|S_{q,L,\blambda}(\bc)\widehat{\CI}_{q}(\widehat{w}^{\pm};\boldsymbol{d})\right|}{q^{n+1}}.$$
	
	\textbf{Step I.} We shall prove that the range $|\bc|> \eta^{-1}Q B^{-\tau+\varepsilon}$ has negligible contribution. The ``simpler" estimate of Corollary \ref{co:simplehardest} yields
	$$\widehat{\CI}_{q}(\widehat{w}^{\pm};\boldsymbol{d})\ll_N \frac{Q}{q}\left(\frac{B^\tau |\bc|\eta}{Q}\right)^{-N}.
 $$ 
	We record from Theorem \ref{thm:sqsum} that uniformly for any $\bc$, whenever $n\geqslant 4$,
	$$\sum_{q=1}^{\infty}\frac{\left|S_{q,L,\blambda}(\bc)\right|}{q^{n+2}}\ll L^n \sum_{q=1}^{\infty}\frac{1}{q^{\frac{n}{2}}}=O(L^n).$$ For any constant $c>0$ and $N>n+1$ we then have the estimate
	\begin{align*}
		&\sum_{q=1}^{cQ}\sum_{\bc:|\bc|>\eta^{-1}Q B^{-\tau+\varepsilon}}\frac{\left|S_{q,L,\blambda}(\bc)\widehat{\CI}_{q}(\widehat{w}^{\pm};\boldsymbol{d})\right|}{q^{n+1}} \\ 
		\ll_N&  \sum_{\bc:|\bc|>\eta^{-1}Q B^{-\tau+\varepsilon}}Q\left(\frac{B^\tau |\bc|\eta}{Q}\right)^{-N}\sum_{q=1}^{\infty}\frac{\left|S_{q,L,\blambda}(\bc)\right|}{q^{n+2}}\\ 
		\ll_N& L^n \eta^{-N}Q^{1+N} B^{-\tau N} \sum_{\bc:|\bc|>\eta^{-1}Q B^{-\tau+\varepsilon}}\frac{1}{|\bc|^N}\\
		\ll_N & L^n \eta^{-N}Q^{1+N} B^{-\tau N}\sum_{t>\eta^{-1}Q B^{-\tau+\varepsilon}}\frac{1}{t^{N-n}}\\ 
  \ll_N& L^n \eta^{-N}Q^{1+N} B^{-\tau N} \left(\eta^{-1}QB^{-\tau+\varepsilon}\right)^{-N+n+1}\\
  \ll_N& L^n \eta^{-n-1} Q^{2+n} B^{-\tau(n+1)+\varepsilon(-N+n+1)}.
	\end{align*}
	We then obtain a power-saving by taking $N$ sufficiently large. 
	
	Nonzero vectors $|\bc|\leqslant \eta^{-1}Q B^{-\tau+\varepsilon}$ with either  $|\widetilde{c_0}|>\frac{\eta^{-1}Q L^2}{B^{1-\varepsilon}}$ or $|\widetilde{c_j}|> \frac{\eta^{-1}QL^2 R}{B^{1-\varepsilon}}$ for some $j\geqslant 2$ can similarly be shown to provide a negligible contribution. Indeed, in the first case $$|d_0|>\frac{\eta^{-1}Q L^2}{B^{1-\varepsilon}}\times \frac{B}{L^2}=\eta^{-1}QB^\varepsilon$$ and in the second case 
	$$|d_j|>\frac{\eta^{-1}QL^2 R}{B^{1-\varepsilon}}\times\frac{B}{L^2R}=\eta^{-1}QB^\varepsilon.$$ Therefore, in either case we have $$|\dd|> \eta^{-1}QB^\varepsilon,$$ and hence the simple estimate in Corollary \ref{co:simplehardest} again gives
	$$\widehat{\CI}_{q}(\widehat{w}^{\pm};\boldsymbol{d})\ll_N \frac{Q}{q} B^{-\varepsilon N}.$$
	Taking $N$ large enough and summing over $q$ and over such $\bc$ as before gives a power-saving.
	
	\textbf{Step II.} 
	In what follows we shall be concerned with  \begin{equation}\label{eq:cond00}
		0\neq|\bc|\leqslant \eta^{-1} Q B^{-\tau +\varepsilon},\quad |\widetilde{c_0}|\leqslant \frac{\eta^{-1}QL^2 }{B^{1-\varepsilon}}\quad \text{and}\quad |\widetilde{c_j}|\leqslant \frac{\eta^{-1}QL^2 R}{B^{1-\varepsilon}} \quad \text{for all }j\geqslant 2.
	\end{equation}
	The ``harder estimate" of Corollary \ref{co:simplehardest} is only useful when the term $A^{-N}\|\widehat{w}^{\pm}\|_N$ can be made small, since otherwise it is better to use the trivial bound $\widehat{\CI}_{q}(\widehat{w}^{\pm};\boldsymbol{d}) \ll 1$. In our setup where $\|\widehat{w}^{\pm}\|_N \asymp \eta^{-N}$, it is natural to choose $A= \eta^{-1} B^{\varepsilon}$. We may apply the ``harder estimate" for all $q< \eta^3B^{\tau-\eps}|\bc|$, because 
 \begin{align*}
     q< \eta^3B^{\tau-\eps}|\bc| \implies q < \eta^3 B^{-\eps}|\bd| \implies \left|\frac{\bd}{q}\right| > \eta^{-3}B^{\eps} 
 \end{align*}
where in the last line we have redefined $\eps$.
Motivated by this, we define	
	\begin{align*}
		\CG^{(1)}(L,R;B)&:= \underset{\substack{\bc\in\Gamma_{\blambda,L},\eqref{eq:cond00}\text{ holds}\\ q\leqslant \eta^3 B^{\tau-\varepsilon}|\bc|^{1-\varepsilon}}}{\sum\sum}\frac{\left|S_{q,L,\blambda}(\bc)\widehat{\CI}_{q}(\widehat{w}^{\pm};\boldsymbol{d})\right|}{q^{n+1}},\\
		\CG^{(2)}(L,R;B)&:= \underset{\substack{\bc\in\Gamma_{\blambda,L},\eqref{eq:cond00}\text{ holds}\\ q> \eta^3 B^{\tau-\varepsilon}|\bc|^{1-\varepsilon}}}{\sum\sum}\frac{\left|S_{q,L,\blambda}(\bc)\widehat{\CI}_{q}(\widehat{w}^{\pm};\boldsymbol{d})\right|}{q^{n+1}}.
	\end{align*}
	
	We assume momentarily that $n\geqslant 5$. By Theorem \ref{thm:sqsum} with partial summation, we have for any real $X>0$,
	\begin{equation}\label{eq:Sqbd}
		\sum_{q\leqslant X} \frac{\left|S_{q,L,\blambda}(\bc)\right|}{q^{l_1}}\ll_\varepsilon |\bc|^\varepsilon L^{n+\varepsilon } X^{\frac{n+5}{2}-l_1+\varepsilon}
	\end{equation}
	for every $l_1\leqslant \frac{n+5}{2}$, and 
	\begin{equation}\label{eq:Sqbd2}
		\sum_{q>X} \frac{\left|S_{q,L,\blambda}(\bc)\right|}{q^{l_2}}\ll_\varepsilon B^\varepsilon |\bc|^\varepsilon L^{n+\varepsilon } X^{\frac{n+5}{2}-l_2+\varepsilon}
	\end{equation}
	for every $l_2>\frac{n+5}{2}$. 
	
	We use \eqref{eq:Sqbd} and the ``harder estimate" for $\widehat{\CI}_{q}(\widehat{w}^{\pm};\boldsymbol{d}) $ to get:
	
	\begin{align*}
		\CG^{(1)}(L,R;B)&\ll \eta^{-n-1} B^{-\frac{n-1}{2}\tau+\varepsilon}\sum_{\bc\in\Gamma_{\blambda,L},\eqref{eq:cond00}\text{ holds}} |\bc|^{-\frac{n-1}{2}}\sum_{q\leqslant \eta^3 B^{\tau-\varepsilon}|\bc|^{1-\varepsilon}}\frac{|S_{q,L,\blambda}(\bc)|}{q^{\frac{n+3}{2}}}\\ &\ll_\varepsilon L^{n+\varepsilon} B^{-\frac{n-3}{2}\tau+\varepsilon}\eta^{2-n}\sum_{\bc\in\Gamma_{\blambda,L},\eqref{eq:cond00}\text{ holds}}|\bc|^{-\frac{n-3}{2}}.
	\end{align*}
	
	Now by \eqref{eq:Sqbd2} and the trivial estimate for $\widehat{\CI}_{q}(\widehat{w}^{\pm};\boldsymbol{d})$, we get:
	
	\begin{align*}
		\CG^{(2)}(L,R;B)&\ll \sum_{\bc\in\Gamma_{\blambda,L},\eqref{eq:cond00}\text{ holds}}\sum_{q> \eta^3 B^{\tau-\varepsilon}|\bc|^{1-\varepsilon}}\frac{\left|S_{q,L,\blambda}(\bc)\right|}{q^{n+1}}\\ &\ll_\varepsilon L^{n+\varepsilon }B^{\varepsilon} \sum_{\bc\in\Gamma_{\blambda,L},\eqref{eq:cond00}\text{ holds}}\left(\eta^3 B^{\tau-\varepsilon}|\bc|^{1-\varepsilon}\right)^{2-\frac{n+1}{2}}\\  &\ll  L^{n+\varepsilon }\eta^{\frac{3(3-n)}{2}} B^{-\frac{n-3}{2}\tau+\varepsilon}\sum_{\bc\in\Gamma_{\blambda,L},\eqref{eq:cond00}\text{ holds}}|\bc|^{-\frac{n-3}{2}}.
	\end{align*}
	
	To estimate the above sums over $\bc$, we use
	\begin{lemma}\label{le:countingcmodL} Assume that $\gcd(\lambda_0,\ldots, \lambda_n,L)=1$. 
		For $T\geq 1$, consider the set
  $$\Theta_{\blambda,L}(T,B):=\left\{\bc\in\BZ^{n+1}:\bc\in\Gamma_{\blambda,L},|\bc|\leqslant T,\max\left(|\widetilde{c_0}|,|\widetilde{c_j}|,2\leqslant j\leqslant n\right)\leqslant \eta^{-1}LB^{\varepsilon}\right\}.$$
		Then for $\varepsilon>0$ small enough, we have
		$$\#\Theta_{\blambda,L}(T,B)\ll_\varepsilon  T \eta^{-n}B^{\varepsilon},$$ 
		uniformly in $L,\blambda$.
	\end{lemma}
	Taking Lemma \ref{le:countingcmodL} for granted, we now apply a dyadic sum over $|\bc|$ up to $\eta^{-1} Q B^{-\tau +\varepsilon}$.  After appropriately redefining $\eps$, we have
	\begin{align*}
		\sum_{\bc\in\Gamma_{\blambda,L},\eqref{eq:cond00}\text{ holds}}|\bc|^{-\frac{n-3}{2}}& \ll_\varepsilon \sum_{T\nearrow \eta^{-1} Q B^{-\tau +\varepsilon}} T^{-\frac{n-3}{2}}\#\Theta_{\blambda,L}(T,B)\\
		&\ll_\varepsilon \eta^{-n}B^\varepsilon \sum_{T\nearrow \eta^{-1} Q B^{-\tau +\varepsilon}}T^{-\frac{n-5}{2}}\ll_\varepsilon \eta^{-n}B^\varepsilon,
	\end{align*} whenever $n\geqslant 5$.
	
	To summarise, we conclude that  $$\CG(L,R;B)\ll_\varepsilon L^{n+\varepsilon}\eta^{\frac{9-5n}{2}} B^{-\frac{n-3}{2}\tau+\varepsilon}.$$
	Finally, using \eqref{eq:IqL2} and \eqref{eq:cihatci}, we have
	\begin{align*}
		\left|I_{q,L,\blambda}(w_{B,R}^{\pm};\bc)\right|&=\frac{1}{|\det M|L^{n+1}}\left|\CI_{q,L}(w_{B,R}^{\pm};\bc)\right|\\ &=\frac{1}{|\det M|}\left(\frac{B}{LR}\right)^{n+1}\left|\widehat{\CI}_{q}(\widehat{w}^{\pm};\boldsymbol{d})\right|,
	\end{align*} and hence
	$$\sum_{q=1}^{\infty}\sum_{\bc\in\BZ^{n+1}\setminus\{\boldsymbol{0}\}}\frac{\left|S_{q,L,\blambda}(\bc)I_{q,L,\blambda}(w_{B,R}^{\pm};\bc)\right|}{(qL)^{n+1}}\ll \frac{1}{L^{n+1}}\left(\frac{B}{LR}\right)^{n+1}\CG(L,R;B).$$
	The proof is thus completed for the case $n\geqslant 5$.

	As for the case $n=4$, we use instead the following estimate from Theorem \ref{thm:sqsum}:
	\begin{equation}\label{n=4 better estimate for Sqc}
		\sum_{q\leqslant X} \left|S_{q,L,\blambda}(\bc)\right|\ll_\varepsilon |\bc|^\varepsilon L^{4+\varepsilon}X^{4+\varepsilon},
	\end{equation}
	valid for all $\bc\in\BZ^5\setminus\{\boldsymbol{0}\}$.
	
	The only alteration required for the case $n=4$ is to replace the partial summation bounds in (\ref{eq:Sqbd}) of \textbf{Step II} with (\ref{n=4 better estimate for Sqc}). We obtain
	\begin{equation*}
		\CG(L,R;B)\ll_\varepsilon L^{4+\varepsilon}\eta^{-7/2}B^{-\tau + \varepsilon} \sum_{\bc\in\Gamma_{\blambda,L},\eqref{eq:cond00}\text{ holds}} |\bc|^{-1}\ll_{\varepsilon} L^{4+\varepsilon}\eta^{-15/2}B^{-\tau+\varepsilon}.
 \end{equation*} Proceding in the same way as before yields the desired bound. \qed
	
	\subsection{Proof of Lemma \ref{le:countingcmodL}}
	Given that $\gcd(\lambda_0,\ldots, \lambda_n,L)=1$, the  lattice $\Gamma_{\blambda,L}$ \eqref{eq:condc} is of rank $(n+1)$ and of determinant $\det (\Gamma_{\blambda,L})$ with $L^n\ll \det (\Gamma_{\blambda,L})\ll L^n$, where the implied constants only depend on $F$. Since it contains $L\be_1,\ldots,L\be_{n+1}$, where $\be_i$ stands for the $i$-th unit vector in $\BZ^{n+1}$, its successive minima satisfy 
	$$\lambda_j\leq L \text{ for all } 1\leqslant j\leqslant n+1,$$ 
	and consequently, by Minkowski's second theorem, we have
	$$\prod_{i=1}^{k}\lambda_i\gg \frac{\det(\Gamma_{\blambda,L})}{L^{n+1-k}} \gg L^{k-1}$$
	for any $1\leq k \leq n+1$. 
	
 First assume that $T> \eta^{-1} QL^2R B^{-1+\varepsilon} = \eta^{-1} LB^{\varepsilon}$. Note that the norm $\max_{0\leqslant i\leqslant n}|\widetilde{c_i}|$ is equivalent to $|\bc|$. Then by standard results on counting lattice points (see e.g. \cite[Lemma 2]{Schmidt}), we obtain
	\begin{align*}
		\#\Theta_{\blambda,L}(T,B)&\ll 1+ \frac{T(L\eta^{-1}B^\varepsilon)^n}{\det(\Gamma_{\blambda,L})}+\sum_{k=1}^{n}\frac{T(L\eta^{-1}B^\varepsilon)^{k-1}}{\prod_{i=1}^{k}\lambda_i}\\ &\ll T\eta^{-n}B^\varepsilon.
	\end{align*}
	Now assume $T\leqslant \eta^{-1} LB^\varepsilon$. We have similarly
	\begin{align*}
		\#\Theta_{\blambda,L}(T,B)&\ll 1+\sum_{k=1}^{n+1}\frac{T^{k}}{\prod_{i=1}^{k}\lambda_i}\ll 1+ T \sum_{k=1}^{n+1} \frac{ (\eta^{-1}LB^{\varepsilon})^{k-1}}{L^{k-1}} \ll T\eta^{-n}B^\varepsilon.
	\end{align*}
	Here we have redefined $\varepsilon$.\qed
	
	\section{The contribution from $\bc=\boldsymbol{0}$}\label{se:c=0}
	In this section we shall derive the main term. In the following, when we write $w_{B,R}$ we mean one of $w_{B,R}^{\pm}$.
	\begin{theorem}\label{thm:c=0}
		Assume that $n\geqslant 4$, and that \eqref{eq:tau} holds with $-1<\tau<1$. Then
		$$\sum_{q=1}^{\infty}\frac{S_{q,L,\blambda}(\boldsymbol{0})I_{q,L,\blambda}(w_{B,R};\boldsymbol{0})}{(qL)^{n+1}}=\frac{B^{n+1}}{(LR)^2}\left(\CI(w_R)\mathfrak{S}(\CW;L,\blambda)+O_{\tau,\varepsilon}\left(\frac{
  B^{\frac{3-n}{4}(1+\tau)+\varepsilon}}{L^n R^{n-1}}\right)\right).$$
  The implied constant is independent of $R,L,\blambda$.
		We refer to \eqref{eq:CIwR} and \eqref{eq:singserW} for the definitions of $\CI(w_R)$ and $\mathfrak{S}(\CW;L,\blambda)$.
	\end{theorem} 
 
 In subsections \S\ref{se:oscint2} \S\ref{se:singser} we assume $n\geqslant 3$.
	\subsection{The singular integrals $I_{q,L,\blambda}(w;\boldsymbol{0})$}\label{se:oscint2}
	 We recall from \eqref{eq:IqL2} that 
	\begin{align*}
		I_{q,L,\blambda}(w_{B,R};\boldsymbol{0})&=\frac{1}{L^{n+1}|\det M|}\int_{\BR^{n+1}} w_{B,R}(M^{-1}\bt)h\left(\frac{q}{Q},\frac{R^2}{B^2}\widetilde{F}(\bt)\right)\operatorname{d}\bt\\ &=\frac{B^{n+1}}{L^{n+1}|\det M|}\CJ_{q}(R;w),
	\end{align*}
	where $$\CJ_{q}(R;w):=\int_{\BR^{n+1}}w_1\left(R\frac{u_j}{u_0},j\geqslant 2\right)w_2(R^2\widetilde{F}(\bu))w_3(u_0)h\left(\frac{q}{Q},R^2\widetilde{F}(\bu)\right)\operatorname{d}\bu.$$
	The remaining part is to analyse $\CJ_{q}(R;w)$ and relate it to the (weighted) singular integral and real density.
	
	\subsubsection{Relation with the volume integral}\label{se:volint} 
  We write $\widetilde{F}(\bu)=u_0^2 G(y_1,\by)$ where $y_i=\frac{u_i}{u_0},1\leqslant i\leqslant n,\by=(y_2,\ldots,y_n)$ and $G$ satisfies $G(y_1,\by)=y_1+F_2\left(y_2,\ldots,y_n\right)$ by Lemma \ref{change of variables}. \begin{theorem}\label{thm:Iq}
		We have
$$I_{q,L,\blambda}(w_{B,R}^{\pm};\boldsymbol{0})=\frac{B^{n+1}}{L^{n+1}R^2}\left(\CI(w_R)+O_{N}\left(\left(\frac{q}{Q}\right)^N\right)\right),$$
  	where we define  $w_R:\BR^{n-1}\to\BR$ by \begin{equation}\label{eq:wR}
  	    w_R(\by)=w_1(Ry_j,2\leqslant j\leqslant n)
  	\end{equation}  and \begin{equation}\label{eq:leadingctsI}
  	    \CI(w_R):= |\det M|^{-1}\int_{\BR\times\BR^{n-1}}w_{R}(\by)u_0^{n-2}w_3(u_0)\left(\frac{\partial G}{\partial y_1}\right)^{-1}(0,\by)\operatorname{d}u_0\operatorname{d}\by. 
  	\end{equation}
	\end{theorem}

\begin{proof}
    	 Observe that \begin{equation}\label{eq:partialF}
		\frac{\partial G}{\partial y_1}(y_1,\by)=1.
	\end{equation}
 Hence, for any such fixed $u_0$ in the domain of integration, $y_1$ can be uniquely solved by e.g. the implicit coordinate $\sigma:=R^2u_0^2 G(y_1,\by)$. So, on defining
	\begin{equation}\label{eq:J1}
		J(R;\sigma)= J(R;0)=\int_{\BR^{n-1}}w_{R}(\by)\operatorname{d}\by = \int_{\BR^{n-1}}w_{R}(\by)\left(\frac{\partial G}{\partial y_1}\right)^{-1}(\sigma,\by)\operatorname{d}\by,
	\end{equation}
 we obtain
	\begin{align*}
		\CJ_q(R;w)&=\int_{\BR^{n+1}}w_R(\by)w_2(R^2u_0^2G(y_1,\by))u_0^n w_3(u_0)h\left(\frac{q}{Q},R^2u_0^2G(y_1,\by)\right)\operatorname{d}u_0\operatorname{d}y_1\operatorname{d}\by\\ &=\frac{1}{R^2}\int_{\BR}u_0^{n-2}w_3(u_0)\operatorname{d}u_0\int_{\BR^{n-1}\times\BR}w_R(\by)w_2(\sigma)h\left(\frac{q}{Q},\sigma\right)\left(\frac{\partial G}{\partial y_1}\right)^{-1}(\sigma,\by)\operatorname{d}\by\operatorname{d}\sigma\\ &=\frac{1}{R^2}\int_{\BR}u_0^{n-2}w_3(u_0)\operatorname{d}u_0\int_{\BR} J(R;\sigma)w_2(\sigma)h\left(\frac{q}{Q},\sigma\right)\operatorname{d}\sigma.
	\end{align*}
	By \cite[Lemma 9]{H-Bdelta}, we have
  \begin{align*}
     \int_{\BR} J(R;\sigma)w_2(\sigma)h\left(\frac{q}{Q},\sigma\right)\operatorname{d}\sigma&=J(R;0)+O_N\left(\left(\frac{q}{Q}\right)^N \right)
 \end{align*}
 where we  recall that $w_2(0)=1$.
	So the leading term of $\CJ_q(R;w)$ equals \begin{equation*}
\begin{split}
    	&|\det M|^{-1}J(R;0)\int_{\BR}u_0^{n-2}w_3(u_0)\operatorname{d}u_0\\ =& |\det M|^{-1}\int_{\BR\times\BR^{n-1}}w_{R}(\by)u_0^{n-2}w_3(u_0)\left(\frac{\partial G}{\partial y_1}\right)^{-1}(0,\by)\operatorname{d}u_0\operatorname{d}\by=\CI(w_R). \qedhere
     \end{split}
	\end{equation*} 
 \end{proof}
 \begin{remark}\label{rmk:realTamagawa}
     In the ``unweighted'' setting, the weight function $w_3$ in the integral $\int_{\BR}u_0^{n-2}w_3(u_0)\operatorname{d}u_0$ is replaced by the condition $\max_{0\leqslant i\leqslant n}|u_i|\leqslant 1$, and \begin{equation}\label{eq:singintcal}
        \int_{\max_{0\leqslant i\leqslant n}|u_i|\leqslant 1} u_0^{n-2}\operatorname{d}u_0=\frac{2}{n-1}\max(1,|y_j|,2\leqslant j\leqslant n)^{-(n-2)}.
     \end{equation} The factor $(n-1)^{-1}$ equals \emph{Peyre's constant} of the quadric, see \cite[Définition 2.4, Lemme 5.4.8]{Peyre}. The term $\max(1,|y_j|)^{-(n-2)}$, together with $|\det M|^{-1}\left(\frac{\partial G}{\partial y_1}\right)^{-1}(0,\by)$, form the density function of the real Tamagawa measure (within the real neighbourhood $V(\BR)\cap (t_0\neq 0)$ with local coordinates $\by=(\frac{t_2}{t_0},\ldots,\frac{t_n}{t_0})$) with respect to our choice of the norm $\|\cdot\|_{\bxi}$ (see \cite[\S5 Définition 5.1, Lemme 5.4.4]{Peyre} where it is termed \emph{Leray measure}). As for $\CI(w_R)$, due to the zoom condition reflected in the weight function $w_R$ \eqref{eq:wR} which prescribes a real region of diameter $O(R^{-1})$, $|y_j|<1,2\leqslant j\leqslant n$ as long as $R$ is large enough. 
Consequently, up to constant, $\CI(w_R)$ is the volume integral weighted by $w_3$ applied on $w_R$.
	 \end{remark}

	\subsubsection{Relation with the singular integrals}\label{se:expsingint}
	In this part we relate the term $\CI(w_R)$ \eqref{eq:leadingctsI} to the singular integrals appearing in the classical circle method, an ``unweighted'' one being $\CI_R$ \eqref{eq:singint}. Via a Fourier inversion procedure (cf. \cite[p. 259]{Birch}), we can write
	\begin{align*}
		J(R;0)=\int_{\BR^2}w_2(\phi)\e\left(\phi z\right)\operatorname{d}\phi\operatorname{d}z\int_{\BR^{n-1}}w_R(\by)\left(\frac{\partial G}{\partial y_1}\right)^{-1}(0,\by)\operatorname{d}\by,
	\end{align*} where for the inner integral $y_1$ is solved by the implicit coordinate $z=G(y_1,\by)$.
	Going back to the coordinate system $\bu$, then back to $\bx:=M^{-1}\bu$, we obtain
	\begin{equation}\label{eq:CIwR}
		\begin{split}
			\CI(w_R) =&|\det M|^{-1}\int_{\BR^{n+1}\times\BR} w_R\left(\frac{u_j}{u_0},j\geqslant 2\right) w_2 (\widetilde{F}(\bu)) w_3(u_0)\e\left(\theta \widetilde{F}(\bu)\right)\operatorname{d}\bu\operatorname{d}\theta\\ =&\int_{\BR^{n+1}\times\BR} w_{R}\left(\frac{t_j}{t_0}(\bx),j\geqslant 2\right) w_2(F(\bx))w_3(t_0(\bx))\e\left(\theta F(\bx)\right)\operatorname{d}\bx\operatorname{d}\theta,
   \end{split}
	\end{equation} where we recall $w_R$ \eqref{eq:wR}.
	We thus re-encounter the classical (weighted) singular integrals whose integration domain is restricted to the one prescribed by our real zoom condition (i.e. by the function $w_R$).

 Conversely, we have
 \begin{equation}\label{eq:ciR}
     \begin{split}
         \CI_R&=|\det M|^{-1}\int_{\substack{\by\in\BR^{n-1}\\ R\by\in U_0}}\left(\frac{\partial G}{\partial y_1}\right)^{-1}(0,\by)\operatorname{d}\by \int_{\max_{0\leqslant i\leqslant n}|u_i|\leqslant 1} u_0^{n-2}\operatorname{d}u_0\\ &= \frac{2}{n-1}|\det M|^{-1}\int_{\substack{\by\in\BR^{n-1}\\ R\by\in U_0}}\frac{1}{\max(1,|y_j|,2\leqslant j\leqslant n)^{(n-2)}}\left(\frac{\partial G}{\partial y_1}\right)^{-1}(0,\by)\operatorname{d}\by.     \end{split} 
     \end{equation} 
By Remark \ref{rmk:realTamagawa}, the term $|\det M|^{-1}\int_{\cdots}$ equals the real Tamagawa measure (with respect to the real metric $\|\cdot\|_{\bxi}$) of the real neighbourhood $\CU(R,\bxi)$.
	\subsection{The singular series $\sum\frac{S_{q,L,\blambda}(\boldsymbol{0})}{q^{n+1}}$}\label{se:singser}
	We recall $\sigma_{p}(\CW;L,\blambda)$ \eqref{eq:sigmapW} and $\mathfrak{S}(\CW;L,\blambda)$ \eqref{eq:singserW}.
	The goal of this part is to prove:
	\begin{theorem}\label{thm:singser}
		 We have
		$$\sum_{q=1}^{\infty}\frac{S_{q,L,\blambda}(\boldsymbol{0})}{q^{n+1}}=L^{2n}\mathfrak{S}(\CW;L,\blambda).$$
	\end{theorem}
	\begin{proof}
		We recall that \begin{align*}
			S_{q,L,\blambda}(\boldsymbol{0})&=\sum_{\substack{a\bmod q\\(a,q)=1}}\sum_{\substack{\bsigma\in(\BZ/qL\BZ)^{n+1}\\ H_{\boldsymbol{\lambda},L}(\boldsymbol{\sigma})\equiv 0\bmod L}}\textup{e}_{qL}\left(a\left(H_{\boldsymbol{\lambda},L}(\boldsymbol{\sigma})+LF(\boldsymbol{\sigma})\right)\right)\\
			&=\sum_{\substack{a\bmod q\\(a,q)=1}}\sum_{\substack{\bsigma\in(\BZ/qL\BZ)^{n+1}\\ F(L\bsigma+\blambda)\equiv 0\bmod L^2}}\e_{qL^2}\left(aF(L\bsigma+\blambda)\right)\\ &=\sum_{\substack{a\bmod q\\(a,q)=1}}\sum_{\substack{\bxi\bmod qL^2\\ \bxi\equiv \blambda\bmod L, F(\bxi)\equiv 0\bmod L^2}}\e_{q}\left(a \frac{F(\bxi)}{L^2}\right).
		\end{align*}
		We now prove that $S_{q,L,\blambda}(\boldsymbol{0})$ is multiplicative, viewed as an exponential sum of modulus $qL^2$. For every fixed $q$, suppose $$qL^2=u_1u_2,$$ with $(u_1,u_2)=1$. Then $L,q$ also admit the decompositions into products of coprime integers $$L=l_1l_2,\quad q=q_1q_2$$ with $l_i=\gcd(u_i,L),q_i=\gcd(u_i,q)$. 
  Note that $$l_i^2=\gcd(u_i,L^2),\quad u_i=q_il_i^2,\quad i=1,2.$$ In the following we write $\overline{u_1}$ for the multiplicative inverse of $u_1$ modulo $u_2$ and $\overline{u_2}$ for the multiplicative inverse of $u_2$ modulo $u_1$. With this we can write $$\bxi=u_1\overline{u_1}\bxi_2+u_2\overline{u_2}\bxi_1
  $$ with $\bxi_i\bmod 
  u_i,i=1,2$. Then $$\bxi\equiv\blambda\mod L\Longleftrightarrow \bxi_i\equiv \blambda\mod l_i,i=1,2,$$ where we use the equality $$u_1\overline{u_1}+u_2\overline{u_2}\equiv 1\mod u_1u_2.$$ Moreover, $$F(\bxi)\equiv 0\mod L^2\Longleftrightarrow F(\bxi_i)\equiv 0\mod l_i^2,i=1,2.$$
		On writing $$a=q_1a_2+q_2a_1$$ with $a_i \mod q_i$, $i=1,2$ and $$\gcd(a,q)=1\Longleftrightarrow \gcd(a_i,q_i)=1,i=1,2,$$ we have
		$$S_{q,L,\blambda}(\boldsymbol{0})=S_1 S_2,$$ where for $i=1,2$, $$S_i=\sum_{\substack{a_i\bmod q_i\\(a_i,q_i)=1}}\sum_{\substack{\bxi_i\bmod u_i\\ \bxi_i\equiv \blambda\bmod l_i, F(\bxi_i)\equiv 0\bmod l_i^2}}\e_{q_i}\left(a_i \frac{u_{3-i}^2\overline{u_{3-i}}^2}{l_{3-i}^2}\frac{F(\bxi_i)}{l_i^2}\right).$$
		Note that $u_{3-i}\overline{u_{3-i}}\equiv 1\bmod u_i$, so it is coprime to $q_i$, and hence a further change of variables $a_i\mapsto a_i \frac{u_{3-i}^2\overline{u_{3-i}}^2}{l_{3-i}^2}$ yields $$S_i=\sum_{\substack{a_i\bmod q_i\\(a_i,q_i)=1}}\sum_{\substack{\bxi_i\bmod u_i\\ \bxi_i\equiv \blambda\bmod l_i, F(\bxi_i)\equiv 0\bmod l_i^2}}\e_{q_i}\left(a_i \frac{F(\bxi_i)}{l_i^2}\right).$$
		
		Consequently, we can write,  $$\sum_{q=1}^{\infty}\frac{S_{q,L,\blambda}(\boldsymbol{0})}{q^{n+1}}=\prod_p \sum_{t=0}^{\infty}\frac{\CS_{p^t}}{p^{(n+1)t}},$$ where 
		\begin{equation}\label{eq:Spt}
			\CS_{p^t}:=\sum_{\substack{a\bmod p^t\\ (a,p)=1}}\sum_{\substack{\bxi\bmod p^{t+2m_p}\\ \bxi\equiv\blambda\bmod p^{m_p},F(\bxi)\equiv 0\bmod p^{2m_p}}}\e_{p^t}\left(a\frac{F(\bxi)}{p^{2m_p}}\right),
		\end{equation}
		with $m_p:=\operatorname{ord}_p(L)$.
		By a standard analysis of Ramanujan sums, we have
\begin{align*}		
  \sum_{t=0}^{k}\frac{\CS_{p^t}}{p^{(n+1)t}} &=&\frac{\#\{\bxi\bmod p^{k+2m_p}:\bxi\equiv \blambda\bmod p^{m_p},F(\bxi)\equiv 0\bmod p^{k+2m_p}\}}{p^{nk}}\\
  &=&\frac{p^{2nm_p}\#\{\bxi\bmod p^{k+2m_p}:\bxi\equiv \blambda\bmod p^{m_p},F(\bxi)\equiv 0\bmod p^{k+2m_p}\}}{p^{n(k+2m_p)}}.
  \end{align*}
		We obtain by taking $k\to\infty$ 
		$$\sum_{t=0}^{\infty}\frac{\CS_{p^t}}{p^{(n+1)t}}=p^{2nm_p}\sigma_{p}(\CW;L,\blambda),$$
		where we recall \eqref{eq:sigmapW}. This finishes the proof of Theorem \ref{thm:singser}. 
	\end{proof}
	
	\subsection{Proof of Theorem \ref{thm:c=0}}
	We divide the summation into ranges $q\leq X$ and $q> X$, where we take $X:=B^{\frac{1}{2}(1+\tau) - \varepsilon}$. 
	
	To deal with the range $q> X$, we note that 
	$$\CI(w_R)\ll R^{-(n-1)}.$$
	Combining this with Theorem \ref{thm:Iq}, we have 
	$$I_{q,L,\blambda}(w_{B,R}^{\pm};\boldsymbol{0})\ll \left(\frac{B}{LR}\right)^{n+1}.$$
	Hence, using Theorem \ref{thm:sqsum}, \begin{align*}
		\sum_{T<q\leq 2T}\frac{S_{q,L,\blambda}(\boldsymbol{0})I_{q,L,\blambda}(w_{B,R}^{\pm};\boldsymbol{0})}{(qL)^{n+1}}&\ll \left(\frac{B}{L^2R}\right)^{n+1}\sum_{T<  q\leq 2T}\frac{|S_{q,L,\blambda}(\boldsymbol{0})|}{T^{n+1}}\\ &\ll_\varepsilon \left(\frac{B}{L^2R}\right)^{n+1} L^{n+\varepsilon} T^{\frac{3-n}{2}+\varepsilon},
	\end{align*}
	and so
	\begin{equation}\label{tail of sqciqc sum}
		\sum_{q> X}\frac{S_{q,L,\blambda}(\boldsymbol{0})I_{q,L,\blambda}(w_{B,R}^{\pm};\boldsymbol{0})}{(qL)^{n+1}}\ll_\varepsilon \frac{B^{n+1}}{R^{n+1} L^{n+2}} X^{\frac{3-n}{2}+\varepsilon}
	\end{equation}
	for any $n\geq 4$. By the same argument, we also have that 
	\begin{equation}\label{tail of sing series}
		\f{B^{n+1}}{(LR)^2}\CI(w_R)\l|\sum_{q \leq X}\f{L^{-2n}S_{q,L, \boldsymbol{}\lambda}(\boldsymbol{0})}{q^{n+1}} -\mathfrak{S}(\CW;L,\blambda)\r|\ll \frac{B^{n+1}}{R^{n+1} L^{n+2}} X^{\frac{3-n}{2}+\varepsilon}.
	\end{equation}
	
	When $q \leq X$, we observe that the error term in Theorem \ref{thm:Iq} decays exponentially. We recall that $(LR)^2\ll B^{1-\tau}$ and hence $Q=\frac{B}{LR}\gg B^{\frac{1}{2}(1+\tau)}$. I.e. if $q\leq B^{\frac{1}{2}(1+\tau)-\varepsilon}$, then we have also $q\ll Q^{1-\varepsilon '}$, for some $\varepsilon' >0$. Applying Theorems \ref{thm:Iq} and \ref{thm:singser}, we obtain 
	\begin{equation}
		\sum_{q\leq X}\f{S_{q,L,\blambda}(\boldsymbol{0})I_{q,L,\blambda}(w_{B,R}^{\pm};\boldsymbol{0})}{(qL)^{n+1}} = \f{B^{n+1}}{(LR)^2}\CI(w_R)\sum_{q \leq X}\f{L^{-2n}S_{q,L, \boldsymbol{}\lambda}(\boldsymbol{0})}{q^{n+1}} + O_N(B^{-\eps N}).
	\end{equation}
	Combining with (\ref{tail of sing series}) and (\ref{tail of sqciqc sum}) and redefining $\eps$, we obtain the desired estimate. \qed

 \subsection{Estimates of local densities} We end this section by establishing the following 
\begin{proposition}\label{prop:growthofcirsigma}
Let $\omega^V_{\bxi,\BR}$ (resp. $\omega^V_{f}$) be the real (resp. finite) Tamagawa measure on $V(\BR)$ (resp. $\CV(\widehat{\BZ})$) induced by the adelic norm \eqref{eq:normbxi} \eqref{eq:heightbxi}, so that $\omega^V_{\bxi}=\omega^V_{\bxi,\BR}\times \omega^V_{f}$. 
   Assume that $n\geqslant 4$. Then we have \begin{equation}\label{eq:cirrealtamagawa}
       \omega^V_{\bxi,\BR} \left(\CU(R,\bxi)\right)\asymp R^{-(n-1)},
   \end{equation} and $$L^{-(n-1)-\varepsilon}\ll_{\varepsilon}\omega^V_{f}\left(\CD(L,\bLambda)\right)\ll L^{-(n-1)},\quad \mathfrak{S}(\CW;L,\bGamma) \asymp L^{-n}.$$
\end{proposition} 
 \begin{proof}
     The arguments in \eqref{eq:ciR} and Remark \ref{rmk:realTamagawa} (see also \cite[\S5.4]{Peyre}) show that \begin{equation}\label{eq:cir2}
         \omega^V_{\bxi,\BR} \left(\CU(R,\bxi)\right)=\frac{n-1}{2}\CI_R, 
     \end{equation} which is clearly $\asymp R^{-(n-1)}$.

     We now estimate the non-archimedean local densities. 
     Let us fix an even integer $L_0\in\BN$ (depending only on $V$) such that $\CV$ (and hence $\CW^o$) is regular over $\BZ/L_0\BZ$. Upon decomposing $\CD(L,\bLambda)$ into a disjoint union of finite adelic neighbourhoods modulo $L_0L$ if necessary, the number of which is $\leqslant L_0^{n+1}=O(1)$, we may proceed under the assumption that $L_0\mid L$.
     For any $p\nmid L$, as $\CV\bmod p$ is smooth,  by the Grothendieck--Lefschetz trace formula, we have, for any prime $l\neq p$
     $$\#\CV(\BF_p)=\sum_{i=0}^{2(n-1)}(-1)^i\operatorname{Tr}(\operatorname{Fr}_p|H^i(\CV_{\overline{\BF_p}},\BQ_l).$$
     By the Lefschetz hyperplane theorem, $H^j(\CV_{\overline{\BF_p}},\BQ_l)=0$ when $j$ is odd, and $H^{2i}(\CV_{\overline{\BF_p}},\BQ_l)=\BQ_l(-i)$ when $1\leqslant i\leqslant n-1$, unless the dimension $n-1=2k$ is even, in which case $H^{2k}(\CV_{\overline{\BF_p}},\BQ_l)=\BQ_l(-k)^{\oplus 2}$.
It follows that
     $$\#\CV(\BF_p)=\begin{cases}
         \frac{p^n-1}{p-1} & n\text{ even};\\ \frac{p^n-1}{p-1}
         \pm p^{\frac{n-1}{2}} & n\text{ odd}.
     \end{cases}
     $$ Let us define the usual $p$-adic density of 
     $V(\BQ_p)$ by $$\sigma_p(V):= \lim_{k\to\infty}\frac{\#\CV(\BZ/p^k\BZ)}{p^{(n-1)k}}.$$ In particular, when $n\geqslant 4$, by Hensel's lemma, $$\sigma_p(V)=\frac{\#\CV(\BF_p)}{p^{\dim V}}=1+p^{-1}+O\left(p^{-\frac{3}{2}}\right). $$
     According to Peyre \cite{Peyre}, \begin{equation}\label{eq:singserV}
		\omega^V_{f}\left(\CD(L,\bLambda)\right)=\prod_{p\mid L}\left(1-\frac{1}{p}\right)\sigma_{p}(\CV;L,\boldsymbol{\Lambda})\times\prod_{p\nmid L}\left(1-\frac{1}{p}\right)\sigma_p(V),
	\end{equation} where  for $p\mid L$, $$\sigma_{p}(\CV;L,\boldsymbol{\Lambda}):=\lim_{k\to\infty}\frac{\#\left(\CV(\BZ/p^k\BZ)\cap \CD_p(L,\bLambda)\right)}{p^{(n-1)k}},$$ is the ``model measure'' of $\CD_p(L,\bLambda)\subseteq V(\BQ_p)$.   By \cite[Theorem 2.13]{Salberger} (since $\CV$ is regular modulo $L$), \begin{equation}\label{eq:sigmapV1}
	    \sigma_{p}(\CV;L,\boldsymbol{\Lambda})=\frac{1}{p^{(n-1)m_p}}.
	\end{equation}
 Hence $$L^{-(n-1)-\varepsilon}\ll_{\varepsilon}\omega^V_{f}\left(\CD(L,\bLambda)\right)=L^{-(n-1)}\prod_{p\mid L}\left(1-\frac{1}{p}\right)\times\prod_{p\nmid L}\left(1+O(p^{-\frac{3}{2}})\right)\ll L^{-(n-1)}.$$

 Recall \eqref{eq:sigmapW} and \eqref{eq:singserW}. We have
	$$\mathfrak{S}(\CW;L,\bGamma)=\prod_{p\mid L}\sigma_{p}(\CW;L,\bGamma)\prod_{p\nmid L}\sigma_p(\CW),$$ where for every $p\nmid L$, $$\sigma_p(\CW):=\lim_{k\to\infty}\frac{\#\CW(\BZ/p^k\BZ)}{p^{nk}}$$ is the usual $p$-adic density of $\CW(\BZ_p)$. The scheme $\CW^o$ is regular modulo $L$. So again by \cite[Theorem 2.13]{Salberger}, for every $p\mid L$, \begin{align*}
		\sigma_{p}(\CW;L,\bGamma)&=\lim_{k\to\infty}\frac{\#\left\{\bv\in\CW^o(\BZ/p^k\BZ):\bv\equiv\bGamma\bmod p^{m_p}\right\}}{p^{nk}}=\frac{1}{p^{nm_p}}.
	\end{align*} 	By \cite[Lemme 5.4.6]{Peyre}, \begin{equation}\label{eq:sigmapV2}
	    \sigma_p(V)=\frac{1-p^{-(n-1)}}{1-p^{-1}}\sigma_p(\CW).
	\end{equation} So, whenever $n\geqslant 4$, $$\sigma_p(\CW)=1+O(p^{-\frac{3}{2}}).$$ Hence $$\mathfrak{S}(\CW;L,\bGamma)=L^{-n}\prod_{p\nmid L}\left(1+O(p^{-\frac{3}{2}})\right)\asymp L^{-n}.$$ The proof is completed.
 \end{proof}
 
	\section{Proof of main theorems}\label{se:proofmainthm}
	\subsection{Proof of Theorem \ref{thm:countingW}}
	This is now evident in view of Theorems \ref{thm:cnot0}, \ref{thm:c=0} and \eqref{eq:Nwdelta}. \qed
	
	\subsection{Proof of Theorem \ref{main counting result} --  Removing the smooth weights}\label{se:removeweight}
	We now compare $\CI(w^{\pm}_R)$ with $\CI_R$, which is defined in \eqref{eq:singint}. 
	
	\begin{lemma}\label{lem:comparing singular integrals}
		Let $w$ be one of the smooth weights $w_{R}^{\pm}$. Assume that
  $$\sup_{u_j\in R^{-1}[a_j,b_j], j\geq 2}  |F_2(u_2,\ldots, u_n)| \leq 1.$$
  Then 
		$$|\CI(w) - \CI_R| \ll \eta \CI_R.$$ 
	\end{lemma}
	
	\begin{proof}
  We define $1_{R}$ to be the characteristic function of the conditions \begin{equation*}	y_j \in R^{-1}[a_j,b_j] \textrm{ for all }2\leq j \leq n,\end{equation*} and we define $1_R^{\pm}$ analogously to $w_R^{\pm}$. Similarly to (\ref{eq:leadingctsI}), we have
  \begin{align*}
      \CI_R&=|\det M|^{-1}\int_{\BR^{n-1}}1_{R}(\by)\left(\frac{\partial G}{\partial y_1}\right)^{-1}(0,\by)\operatorname{d}\by\int_{|u_0|\leqslant 1}u_0^{n-2}\operatorname{d}u_0\\ &= |\det M|^{-1}\int_{\substack{\BR\times \BR^{n-1}\\ |u_0|\leqslant 1}}u_0^{n-2}1_R(\by)\op{d}u_0\op{d}\by.
  \end{align*}
   The measure of the set of $\by\in \BR^{n-1}$ for which $w_R^{\pm}(\by)\neq 1_R^{\pm}(\by)$ is $O(\eta R^{1-n})$, and the measure of the set of $u_0\in \BR$ for which $w_3^{\pm}(u_0)\neq 1_{[-1,1]}(u_0)$ is $O(\eta)$. We deduce from \eqref{eq:leadingctsI} that 
		\begin{align*}
			\CI(w) &= |\det M|^{-1}\int_{\BR^{n-1}}w_R^{\pm}(\by)\op{d}\by \int_{\BR}u_0^{n-2}w_3^{\pm}(u_0) \op{d}u_0\\&= |\det M|^{-1} \l(\int_{\BR^{n-1}}1_R(\by) \op{d}\by + O(\eta R^{1-n})\r)\l(\int_{|u_0|\leq 1}u_0^{n-2}\op{d}u_0 + O(\eta)\r)\\
			&=\CI_R+ O(\eta R^{1-n})\\
			&=\CI_R(1+O(\eta)) \qedhere.
		\end{align*}
	\end{proof}
	We now come to the proof of Theorem \ref{main counting result}, which gives an asymptotic formula for $ \CN_{\CW}((R,\bxi),(L,\bGamma);B)$ under the assumption that $0<\tau<1$, where
 $$(LR)^2=O(B^{1-\tau}).$$
 For the proof of Theorem \ref{thm:countingV} we also need an upper bound for $ \CN_{\CW}((R,\bxi),(L,\bGamma);B)$, where $\tau$ is allowed to be negative. We will deduce the following proposition at the same time as we prove Theorem \ref{main counting result}.

 \begin{proposition}\label{propupperbound}
Let $n\geq 4$, $L,R,B\geq 1$, $L\in \mathbb{N}$, $R,B\in \mathbb{R}$ and assume that $(LR)^2=O(B^{1-\tau})$ for some $-1<\tau <1$. Let $\bGamma\in\CW^o(\BZ/L\BZ)$. Then we have
 $$\CN_{\CW}((R,\bxi),(L,\bGamma);B)\ll_\tau  \f{B^{n-1+\varepsilon}}{L^nR^{n-1}}(E_{\tau}(1;B) +1)+ \frac{B^{\frac{(n+1)(1-\tau)}{2}+\varepsilon}}{L^{n+1}R^n},$$
 where $E_{\tau}(1;B)$ is defined as in Theorem \ref{thm:countingW}.

 \end{proposition}

 For the following assume that we are given a value $-1<\tau<1$. Let $1_{B,R}$ be the indicator function for the conditions
	$$ \f{t_j}{t_0} \in R^{-1}[a_j,b_j] \textrm{ for all }2 \leq j\leq n, \quad |t_0| \leq B.$$
 Assume that we are given a point $\bx \in \BR^{n+1}$ with $F(\bx)=0$, and $1_{B,R}(\bx)=1$. Then we find that
$$|t_j|\leq R^{-1} |t_0| \max\{|a_j|, |b_j|\} \leq B R^{-1} \max\{|a_j|, |b_j|\}\leq B , \quad j\geq 2$$
under the assumption that $R^{-1} \max_{2\leq j\leq n} \{|a_j|,|b_j|\} \leq 1$. Moreover, we have
$$F(\bx)=t_0t_1+F_2(t_2,\ldots, t_n)$$
where $\bft= M \bx$. In particular, if $t_0\neq 0$, we obtain that
$$|t_1|\leq |t_0|^{-1} |F_2(t_2,\ldots, t_n)| \leq |t_0| \sup_{ u_j \in R^{-1}[a_j,b_j], j\geq 2} |F_2(u_2,\ldots, u_n)| \leq |t_0|\leq B$$
under the assumption that $\sup_{ u_j \in R^{-1}[a_j,b_j], j\geq 2} |F_2(u_2,\ldots, u_n)| \leq 1$. 
Note that for all $\bx \in \BR^{n+1}$ with $F(\bx)=0$ and $t_0\neq 0$, we have
	$$w^{-}_{B,R}(\bx) \leq 1_{B,R}(\bx) \leq \sum_{k=0}^{\infty} w^{+}_{B\eta^k,R}(\bx),$$
	and so \begin{align*}
	    \CN_{\CW}(w_{B,R}^-;(L,\bGamma))&\leq \CN_{\CW}(1_{B,R};(L,\bGamma))\\&=\CN_{\CW}((R,\bxi),(L,\bGamma);B) \leq \sum_{k=0}^{\infty}\CN_{\CW}(w_{B\eta^k,R}^+;(L,\bGamma)).
	\end{align*}

We let $R_0$ be sufficiently large such that
\begin{equation}\label{sizeR0}
\sup_{u_j\in R_0^{-1}[a_j,b_j], j\geq 2}  |F_2(u_2,\ldots, u_n)| \leq 1,\quad \mbox{ and } \quad R_0^{-1} \max_{2\leq j\leq n} \{|a_j|,|b_j|\} \leq 1,
\end{equation}

and we assume $R\geq R_0$ in the following.

\begin{lemma}\label{le:removingsmoothweights} 
Suppose that $\eta \gg B^{-1/2}$ and $\eta(1+\eta)<\frac{1}{2}$. Then, for $0<\tau<1$, we have 
\begin{multline*}
    \sum_{k=0}^{\infty}\CN_{\CW}(w_{B\eta^k,R}^+;(L,\bGamma))\\ = \CN_{\CW}(w_{B,R}^+;(L,\bGamma))  + O_\tau \left( \eta \mathfrak{S}(\CW;L,\bGamma) \frac{ B^{n-1}}{R^{n-1}} + \eta  \frac{ B^{n-1}}{L^n R^{n-1}}B^\varepsilon E_{\tau}(\eta;B)\right).
\end{multline*}
For $-1<\tau\leq 0$ we have
\begin{multline*}
    \sum_{k=0}^{\infty}\CN_{\CW}(w_{B\eta^k,R}^+;(L,\bGamma))\\ = \CN_{\CW}(w_{B,R}^+;(L,\bGamma))  + O_\tau \left( \frac{ B^{n-1}}{L^n R^{n-1}}B^\varepsilon E_{\tau}(\eta;B) + \frac{B^{\frac{(n+1)(1-\tau)}{2}+\varepsilon}}{L^{n+1}R^n}\right).
\end{multline*}
\end{lemma}

\begin{proof}
We observe that the supports of the weight functions $w^+_{B\eta^k, R}$ overlap, but each $\bx\in \mathbb{R}^{n+1}$ is contained in at most two of them because of the condition $\eta (1+\eta) <\frac{1}{2}$. 
Therefore,  $$\sum_{k:B\eta^k\leqslant 1}\CN_{\CW}(w_{B\eta^k,R}^+;(L,\bGamma))\leqslant 2\#\{\bx\in\BZ^{n+1}:\|\bx\|_{\bxi}\leqslant C_0\}=O(1),$$
for a constant $C_0>0$ which only depends on $F_2$ and the intervals $[a_j,b_j]$, $j\geq 2$.

From now on we focus on $k \geqslant 1$ such that $B\eta^k >1$. Let $$\tau_0:=\tau,$$ and define $\tau_k$ by the equation
\begin{equation}\label{eq:tauk}
B^{1-\tau_0} = (B\eta^k)^{1-\tau_k}.
\end{equation}

Note that by assumption we have $-1<\tau_0<1$ and hence $\tau_k\leq \tau_0$ for all such $k$ in consideration. Let $\theta$ be a small parameter to be chosen later. If we assume moreover that $-1+\theta<\tau_k<1$, then we observe that
$$B\eta^k = B^{\frac{1-\tau_0}{1-\tau_k}} \geq B^{\frac{1-\tau_0}{2}}.$$

The dependence of the implicit constants in Theorem \ref{thm:countingW} is continuous and so in the compact interval $[-1+\theta,\tau_0]$ they depend only on $\theta$ and $\tau_0$. Now Theorem \ref{thm:countingW} implies that
\begin{equation}\label{thm 3.3 negative tau}
\CN_{\CW}(w_{B\eta^k,R}^+;(L,\bGamma)) \ll_{\varepsilon,\tau_0,\theta} \mathfrak{S}(\CW;L,\bGamma) \frac{(B\eta^k)^{n-1}}{R^{n-1}}+\frac{(B\eta^k)^{n-1}}{L^nR^{n-1}}B^\varepsilon E_k(\eta;B),
\end{equation}
for any $k$ such that $-1+\theta\leq\tau_k\leq \tau_0$, where
$$ E_k(\eta;B) := 
\begin{cases}
\eta^{\frac{9-5n}{2}}(B\eta^k)^{\frac{-(n-3)}{2}\tau_k}, &\textrm{ if }n\geq 5;\\
\eta^{-15/2}(B\eta^k)^{-\tau_k} + (B\eta^k)^{-\frac{1+\tau_k}{4}}, &\textrm{ if }n=4.
\end{cases}
$$
It follows from \eqref{eq:tauk} that $(B\eta^k)^{-\tau_k} = B^{-\tau_0}\eta^{-k}$, whence
$$E_k(\eta;B) \ll (\eta^{-\frac{k(n-3)}{2}} +\eta^{-k}) E_0(\eta;B).$$ Therefore,
\begin{align}
&\sum_{\ss{k:B\eta^k>1\\ -1+\theta<\tau_k<1\\ k\geq 1}}\CN_{\CW}(w_{B\eta^k,R}^+;(L,\bGamma)) \nonumber\\ \ll_{\tau_0,\varepsilon,\theta} & \sum_{\ss{k:B\eta^k>1\\ -1+\theta<\tau_k<1\\ k\geq 1}}\left(\mathfrak{S}(\CW;L,\bGamma)  \frac{(B\eta^k)^{n-1}}{R^{n-1}}+ \frac{(B\eta^k)^{n-1}}{L^nR^{n-1}}B^\varepsilon E_k(\eta; B)\right)\nonumber\\
\ll_{\tau_0,\varepsilon,\theta} & \mathfrak{S}(\CW;L,\bGamma) \frac{\eta B^{n-1}}{R^{n-1}}+\frac{ B^{n-1}}{L^n R^{n-1}}B^\varepsilon E_0(\eta; B)\sum_{\ss{k\geq 1}}(\eta^{k(n-2)}+\eta^{k(n/2+1/2)} )\\
\ll_{\tau_0,\varepsilon,\theta} & \eta \mathfrak{S}(\CW;L,\bGamma) \frac{ B^{n-1}}{R^{n-1}} + \eta  \frac{ B^{n-1}}{L^n R^{n-1}}B^\varepsilon E_{\tau}(\eta;B)\label{eq:small values of k}.
\end{align}

Suppose that $k$ is chosen such that $\tau_k\leqslant -1+\theta$. If $\eta^k \geq B^{-1}$, then we have
$$B^{1-\tau_0} = (B\eta^k)^{1-\tau_k} \geqslant (B\eta^k)^{2-\theta} = B^{2-\theta} \eta^{(2-\theta)k},$$
and hence $\eta^k \ll B^{-\frac{\tau_0+1-\theta}{2-\theta}}$.

Therefore, there is a constant $C>0$ such that 
$$\sum_{\ss{k:B\eta^k>1\\ \tau_k \leq -1+\theta}}\CN_{\CW}(w_{B\eta^k,R}^+;(L,\bGamma)) \ll \CN_{\CW}(1_{C B^{\varphi},C^{-1}R};(L,\bGamma)). $$  
Here, $\CN_{\CW}(1_{C B^{\varphi},C^{-1}R};(L,\bGamma))$ is almost the same as our original counting function, but with the height only going up to $C B^{\varphi}$ instead of $B$, where $$\varphi=\varphi(\tau,\theta):=\frac{1-\tau_0}{2-\theta}.$$
In order to bound it, we ignore the condition $F(\bx)=0$, and keep only the congruence condition $\bx \equiv \bGamma \bmod{L}$, and the real zoom condition (together with its consequence for the first Witt coordinate):
\begin{equation}\label{eq:recap of zoom conds}
    \left|\frac{t_j}{t_0}\right| \ll R^{-1} \textrm{ for all } 2\leq j \leq n,\quad \left|\frac{t_1}{t_0}\right|\ll R^{-2}.
\end{equation}
We obtain
\begin{equation}
	\CN_{\CW}(1_{CB^{\varphi},C^{-1}R};(L,\bGamma)) \ll \#\{\bx \in \BZ^{n+1}: \bx \equiv \boldsymbol{\Gamma} \bmod{L}, (\ref{eq:recap of zoom conds}) \textrm{ holds}, \|\bx\|_{\bxi} \leq C B^{\varphi}\}.
	\end{equation}
We proceed similarly to the proof of Lemma \ref{le:countingcmodL}. Let $\Lambda$ be the lattice $\bx \equiv \boldsymbol{0} \mod{L}$, and observe that the lattice $\Lambda$ has determinant $L^{n+1}$ and successive minima satisfying $\lambda_i\gg L$ for all $i$.
Therefore, by standard lattice point counting results (see e.g. \cite[Lemma 2]{Schmidt}), we have, on recalling the assumption $(LR)^2=O(B^{1-\tau})$,
\begin{equation}\label{eq:bdCWBR}
    \CN_{\CW}(1_{C B^{\varphi},C^{-1}R};(L,\bGamma))  \ll 1+ \frac{B^{(n+1)\varphi}}{L^{n+1}R^n} + \sum_{k=1}^n \frac{B^{k\varphi}}{L^kR^{k-1}} \ll \frac{B^{{(n+1)\varphi}}}{L^{n+1}R^n},
\end{equation}
where the last inequality follows from noting that $\varphi>\frac{1-\tau}{2}$ and hence
$$\frac{B^{\phi}}{LR}\gg 1.$$

We now make a case distinction. For $0<\tau<1$ and $n\geq 4$ and $\theta\leqslant \tau$, we have $\varphi\geqslant\frac{1}{2}$ and
$$ \frac{B^{{(n+1)\varphi}}}{L^{n+1}R^n} \ll \frac{B^{\frac{n+1}{2}}}{L^{n+1}R^n} \ll B^{-1/2}\left(\frac{B^{n-1}}{L^nR^{n-1}}\right) \ll \eta \mathfrak{S}(\CW;L,\bGamma) \frac{ B^{n-1}}{R^{n-1}},$$
provided that $\eta \gg B^{-1/2}$. Note that here we used that $F$ is a non-singular quadratic form and that $\gcd(L,\mathbf{\Gamma})=1$. Once combined with (\ref{eq:small values of k}), this completes the proof of the lemma for the case $0<\tau<1$.\par
In the case $-1<\tau\leq 0$ the lemma follows directly from the estimate in equation \eqref{eq:bdCWBR} after taking $\theta$ sufficiently small in terms of $\tau$ and $\varepsilon$.
\end{proof}

Next we use the final lattice point estimates from the proof of Lemma \ref{le:removingsmoothweights} to deduce an upper bound for $\CN_{\CW}((R,\bxi),(L,\bGamma);B)$, which holds for all values of $B,R,L$ under consideration.

\begin{lemma}\label{latticeupperbound}
Let $B, R,L\geq 1$ and $\bGamma\in\CW^o(\BZ/L\BZ)$. Then we have the upper bound

$$\CN_{\CW}((R,\bxi),(L,\bGamma);B) \ll 1 + \frac{B^{n+1}}{L^{n+1}R^n} + \sum_{k=1}^n \frac{B^k}{L^kR^{k-1}}$$
\end{lemma}

\begin{proof}
We start with the observation that
$$\CN_{\CW}((R,\bxi),(L,\bGamma);B) \ll \#\{\bx \in \BZ^{n+1}: \bx \equiv \boldsymbol{\Gamma} \bmod{L}, (\ref{eq:recap of zoom conds}) \textrm{ holds}, \|\bx\|_{\bxi} \leq  B\} .$$
Hence it is sufficient to count the number of lattice points in the shifted lattice $\Lambda$ as in the proof of Lemma \ref{le:removingsmoothweights} in the box given by
$$|t_0|\ll B,\quad |t_i|\ll R^{-1}B, \quad 1 \leq i \leq n.$$
Now, we again use a standard lattice point counting result such as in \cite[Lemma 2]{Schmidt}. Note that if $\bfx$ and $\bfx'$ are both counted above, then the difference $\bfx-\bfx'$ is contained in the lattice $\Lambda$.

\end{proof}

\begin{proof}[Completion of proof of Theorem \ref{main counting result} and Proposition \ref{propupperbound}] We start with the case $0<\tau<1$. The constants $\CI(w_R^{\pm})$ obtained from Theorem \ref{thm:countingW} can be replaced with $\CI_R(1+O(\eta))$ by using Lemma \ref{lem:comparing singular integrals}. We conclude from Lemma \ref{le:removingsmoothweights} and Proposition \ref{prop:growthofcirsigma} that for $\eta\gg B^{-1/2}$ we have
	\begin{align*}
	    \CN_{\CW}((R,\bxi),(L,\bGamma);B)&= B^{n-1}\CI_R\mathfrak{S}(\CW;L,\bGamma)(1+O(\eta)) + O_{\varepsilon,\tau} \left(\f{B^{n-1+\varepsilon}}{L^nR^{n-1}}E_{\tau}(\eta;B)\right)\\ &= B^{n-1}\CI_R\mathfrak{S}(\CW;L,\bGamma)\l(1+ O(\eta) + O_{\varepsilon,\tau}(B^\varepsilon E_{\tau}(\eta;B)) \r),
	\end{align*}
 where
	$E_\tau(\eta;B)$ is given by \eqref{eq:Etau}.
	
	The error term above is minimised by choosing $\eta$ such that $\eta = E_{\tau}(\eta;B)$. When $n\geq 5$, this gives $\eta = B^{\f{-(n-3)\tau}{5n-7}}$
 , and when $n=4$, this gives $\eta = B^{-2\tau/17}$. This provides the power-saving error terms claimed in Theorem \ref{main counting result}.\par
 In the case $-1<\tau\leq 0$, by choosing $\eta \sim 1$, we obtain
 $$\CN_{\CW}((R,\bxi),(L,\bGamma);B)\ll_\tau  \f{B^{n-1+\varepsilon}}{L^nR^{n-1}}E_{\tau}(1;B) + \frac{B^{\frac{(n+1)(1-\tau)}{2}+\varepsilon}}{L^{n+1}R^n},$$
 which completes the proof of Proposition \ref{propupperbound}. \end{proof}

	\subsection{Proof of Theorem \ref{thm:countingV} -- From the affine cone to its projective image}\label{se:proofofmainthmprojquadric}
	 Since $\dim V\geqslant 3$, $\operatorname{rank}\operatorname{Pic}(V)=1$, and so $W^o$ is  the unique universal torsor of $V$ (up to isomorphism). 
  
In the following we assume that $R\geq R_0$, where $R_0$ is defined in \eqref{sizeR0}. For every $\bGamma\in\CW^o(\BZ/L\BZ)$, let us define the counting function $ \CN_{\CW^o}((R,\bxi),(L,\bGamma);B)$ for $\CW^o$ with the local conditions $(R,\bxi),(L,\bGamma)$ to be \begin{align*}
	&\sum_{\substack{\bx\in\CW^o(\BZ)\\ \bx\equiv\bGamma\bmod L}}1_{B,R}(\bx)\\&=\#\{\bx\in\BZ^{n+1}\setminus\boldsymbol{0}:F(\bx)=0,\pi(\bx)\in \CU(R,\bxi),\bx\equiv\bGamma\bmod L,\gcd_{0\leqslant i\leqslant n}(x_i)=1,\|\bx\|_{\bxi}\leqslant B\}.
	\end{align*}
	 It is related to the corresponding counting function for $\CW$ via a Möbius inversion:
	 $$\CN_{\CW^o}((R,\bxi),(L,\bGamma);B)=\sum_{d\in\BN}\mu(d)\CN_{\CW}((R,\bxi),(L,\overline{d}\bGamma);B/d).$$
	 where we observe that whenever $d\bx\equiv \bGamma \bmod L$, we have $(d,L)\mid \bGamma$. So the $d$-sum is restricted to $(d,L)=1$, and for every such $d$, we write $\overline{d}$ for the multiplicative inverse modulo $L$. 
	
  	Now for every $p\mid L$, we write $m_p:=\operatorname{ord}_p(L)$ and define \begin{equation}\label{eq:Ep}
  \CE_p(L,\bGamma):=\{y_p\in\CW^o(\BZ_p): y_p\equiv \bGamma\bmod p^{m_p}\}\subseteq \CW^o(\BZ_p),
  	\end{equation} and $$\CE(L,\bGamma):= \prod_{p\mid L}\CE_p(L,\bGamma)\times\prod_{p\nmid L}\CW^o(\BZ_p)\subseteq \prod_{p<\infty}\CW^o(\BZ_p),$$  which is a non-empty finite adelic open set  of $\CW^o$.	
  
  Let us take $\bGamma\in\CW^o(\BZ/L\BZ)$ such that $\pi(\bGamma)\equiv \bLambda\bmod L$. We then have \begin{equation}\label{eq:decompadelic}
		\pi^{-1}(\CD(L,\bLambda))=\bigsqcup_{\gamma\in \BG_{\operatorname{m}}(\BZ/L\BZ)}\CE(L,\gamma \bGamma). \end{equation}
  By \eqref{eq:decompadelic} we can now relate the counting functions of $\CW^o$ and $\CW$ to that of $V$ \eqref{eq:CNV}:
	\begin{align*}
		\CN_{V}((R,\bxi),(L,\bLambda);B)&=\frac{1}{2}\sum_{\gamma\in \BG_{\operatorname{m}}(\BZ/L\BZ)}\CN_{\CW^o}((R,\bxi),(L,\gamma\bGamma);B)\\ &=\frac{1}{2}\sum_{\gamma\in \BG_{\operatorname{m}}(\BZ/L\BZ)}\sum_{\substack{(d,L)=1\\ d\ll B}}\mu(d)\CN_{\CW}((R,\bxi),(L,\gamma\overline{d}\bGamma);B/d).
	\end{align*}
The lemma below allows to extend the $d$-sum to infinity.
\begin{lemma}\label{le:dsum}
We have
    \begin{equation*}
\begin{split}
    \CN_{V}((R,&\bxi),(L,\bLambda);B)\\
    &=B^{n-1}\frac{\CI_R}{2}\sum_{\gamma\in \BG_{\operatorname{m}}(\BZ/L\BZ)}\sum_{(d,L)=1 }\frac{\mu(d)}{d^{n-1}}\mathfrak{S}(\CW;L,\gamma\overline{d}\bGamma)+ O_{\eps}\l(\f{B^{n-1+\eps}}{L^{n-1}R^{n-1}}E_{\tau}(B)\r).
    \end{split}
\end{equation*}
\end{lemma}
\begin{proof}
    We split the summation over $d$ into different ranges. For $1 \leq d \ll B$, we define real numbers $\tau_d$ such that 
\begin{equation}\label{def of taud}
    (LR)^2 \asymp \l(\f{B}{d}\r)^{1-\tau_d}.
\end{equation}
Note that we can assume $\tau_d <1$ for all $d$ under consideration. Furthermore, we have $\tau_d>\theta'$ for some fixed small constant $\theta'>0$, provided that $d\ll B^{\tau-\eps}$ for some $\eps>0$, and hence for these values of $d$ we are in a position to apply Theorem \ref{main counting result}. Note that we used again here that the dependence of the implicit constants is continuous in $\tau$ in Theorem \ref{main counting result}. Motivated by this, we define $D_1 = B^{\tau-\eps}$ and we write
$$\Sigma_{\leq D_1}:= \frac{1}{2}\sum_{\gamma\in \BG_{\operatorname{m}}(\BZ/L\BZ)}\sum_{\substack{(d,L)=1\\ d\leq D_1}}\mu(d)\CN_{\CW}((R,\bxi),(L,\gamma\overline{d}\bGamma);B/d).$$
We also let $E_{\tau_d}(\cdot)$ be as defined in (\ref{EB}). With this notation in hand, Theorem \ref{main counting result} implies that 
	\begin{align}
		\Sigma_{\leq D_1}=B^{n-1}\frac{\CI_R}{2}\sum_{\gamma\in \BG_{\operatorname{m}}(\BZ/L\BZ)}\sum_{\substack{(d,L)=1\\ d\leq D_1}}\left(\frac{\mu(d)}{d^{n-1}}\mathfrak{S}(\CW;L,\gamma\overline{d}\bGamma)(1+O_\varepsilon(B^\varepsilon E_{\tau_d}(B/d)))\right)\nonumber\\
  =B^{n-1}\frac{\CI_R}{2}\sum_{\gamma\in \BG_{\operatorname{m}}(\BZ/L\BZ)}\sum_{\substack{(d,L)=1\\ d\leq D_1}}\frac{\mu(d)}{d^{n-1}}\mathfrak{S}(\CW;L,\gamma\overline{d}\bGamma)+O_\varepsilon\left( \frac{B^{n-1+\varepsilon}}{(LR)^{n-1}}  \sum_{d\leq D_1} \frac{E_{\tau_d}(B/d)}{d^{n-1}} \right) \label{eq Moebius up to D1}
	\end{align}
In view of (\ref{def of taud}), the condition $(LR)^2 \ll B^{1-\tau}$ may be alternatively written as   
\begin{equation}\label{eq new form of taud}
    \l(\f{B}{d}\r)^{-\tau_d} \ll dB^{-\tau}. 
\end{equation}
Let $\theta = (n-3)/(5n-7)$ if $n\geq 5$ and $2/17$ if $n=4$, so that $E_{\tau}(B) = B^{-\tau\theta}$. Using (\ref{eq new form of taud}), we obtain
\begin{equation}\label{eq Etaud sum}
    \sum_{d\leq D_1} \f{E_{\tau_d}(B/d)}{d^{n-1}} \ll \sum_{d\leq D_1} \f{(dB^{-\tau})^{\theta}}{d^{n-1}} \ll B^{-\tau \theta} \ll E_{\tau}(B). 
\end{equation}

In order to complete the summation over $d$, we note that
\begin{equation}\label{eq completing d sum}
  \sum_{\gamma\in \BG_{\operatorname{m}}(\BZ/L\BZ)}\sum_{\ss{(d,L)=1 \\ d > D_1}}\frac{\mu(d)}{d^{n-1}}\mathfrak{S}(\CW;L,\gamma\overline{d}\bGamma) \ll \f{B^{n-1}}{(LR)^{n-1}}\sum_{d>B^{\tau - \eps}}d^{1-n} \ll \f{B^{n-1 + \eps}}{(LR)^{n-1}}B^{(2-n)\tau}.\\
\end{equation}
Inserting the estimates (\ref{eq Etaud sum}), (\ref{eq completing d sum}) into (\ref{eq Moebius up to D1}), and observing that $B^{(2-n)\tau} \ll E_{\tau}(B)$, we obtain 
\begin{equation}\label{eq sum over small d}
    \Sigma_{\leq D_1} = B^{n-1}\frac{\CI_R}{2}\sum_{\gamma\in \BG_{\operatorname{m}}(\BZ/L\BZ)}\sum_{(d,L)=1 }\frac{\mu(d)}{d^{n-1}}\mathfrak{S}(\CW;L,\gamma\overline{d}\bGamma) + O_{\varepsilon}\l(\f{B^{n-1 + \varepsilon}}{L^{n-1}R^{n-1}}E_{\tau}(B)\r).
\end{equation}

Next, we let $D_2 = B^{\f{1+\tau}{2} - \eps}$. Then for any $d\leq D_2$, we have $\tau_d >-1+\theta'$ for some positive $\theta'$, and so Proposition \ref{propupperbound} applies. Therefore, the quantity 
$$\Sigma_{D_1< d\leq D_2}:= \frac{1}{2}\sum_{\gamma\in \BG_{\operatorname{m}}(\BZ/L\BZ)}\sum_{\substack{(d,L)=1\\ D_1< d\leq D_2}}\mu(d)\CN_{\CW}((R,\bxi),(L,\gamma\overline{d}\bGamma);B/d)$$
satisfies the estimate
\begin{equation}\label{eq intermediate range of d}
    \Sigma_{D_1< d\leq D_2} \ll \sum_{\gamma\in \BG_{\operatorname{m}}(\BZ/L\BZ)}\sum_{\substack{(d,L)=1\\ D_1< d\leq D_2}} \left(\f{(B/d)^{n-1+\varepsilon}}{L^nR^{n-1}}(E_{\tau_d}(1;B/d) +1)+ \frac{(B/d)^{(n+1)(1-\tau_d)/2+\varepsilon}}{L^{n+1}R^n}\right).
\end{equation}
First, we compute 
\begin{align*}
    &\sum_{\gamma\in \BG_{\operatorname{m}}(\BZ/L\BZ)}\sum_{\substack{(d,L)=1\\ D_1< d\leq D_2}} \frac{(B/d)^{(n+1)(1-\tau_d)/2}}{L^{n+1}R^n}\\ &\ll \sum_{D_1< d\leq D_2} \frac{(B/d)^{(n+1)(1-\tau_d)/2}}{(LR)^n}\ll \sum_{D_1< d\leq D_2}\f{(LR)^{n+1}}{(LR)^n} \ll D_2LR\ll B^{\f{1+\tau}{2} - \eps}B^{\f{1-\tau}{2}} \ll B.
\end{align*}
For the remaining terms in (\ref{eq intermediate range of d}), very crude estimates suffice. We recall that 
\begin{equation*}
    E_{\tau_d}(1;B/d) = \begin{cases}
        (B/d)^{-\f{n-3}{2}\tau_d}, &\textrm{ if }n\geq 5;\\
        (B/d)^{-\tau_d} + (B/d)^{-\f{1+\tau_d}{4}}, &\textrm{ if }n=4.
    \end{cases}
\end{equation*}
Using (\ref{eq new form of taud}), for $\tau_d<0$ and any $n\geq 4$, we have 
\begin{equation}
    E_{\tau_d}(1;B/d) \ll (B/d)^{-\f{n-2}{2}\tau_d} \ll (dB^{-\tau})^{\f{n-2}{2}} \ll d^{\f{n-2}{2}}.
\end{equation}
On the other hand, if $\tau_d>0$, then clearly $E_{\tau_d}(1;B/d) \ll 1$, and so the above estimate holds in any case. Plugging into (\ref{eq intermediate range of d}), we find that 
\begin{align*}
\sum_{\gamma\in \BG_{\operatorname{m}}(\BZ/L\BZ)}\sum_{\substack{(d,L)=1\\ D_1< d\leq D_2}} \f{(B/d)^{n-1+\varepsilon}}{L^nR^{n-1}}(E_{\tau_d}(1;B/d) +1) &\ll \f{B^{n-1}}{(LR)^{n-1}}\sum_{D_1< d\leq D_2}\f{d^{\f{n-2}{2}}}{d^{n-1}} \\
&\ll \f{B^{n-1}}{(LR)^{n-1}}D_1^{1-\f{n}{2}}.
\end{align*}
Since $D_1 = B^{\tau -\eps}$ and $n\geq 4$, it is clear that this term is dominated by the error term already present in (\ref{eq sum over small d}). 

Finally, we consider the range $D_2<d \ll B$ and set
$$\Sigma_{d>D_2}:= \frac{1}{2}\sum_{\gamma\in \BG_{\operatorname{m}}(\BZ/L\BZ)}\sum_{\substack{(d,L)=1\\ d> D_2}}\mu(d)\CN_{\CW}((R,\bxi),(L,\gamma\overline{d}\bGamma);B/d).$$
For this we use Lemma \ref{latticeupperbound} and find that
\begin{align*}
&\Sigma_{d>D_2}\ll L \sum_{d>D_2} \left(  1 + \frac{\left(\frac{B}{d}\right)^{n+1}}{L^{n+1}R^n} + \sum_{k=1}^n \frac{\left(\frac{B}{d}\right)^k}{R^{k-1}L^k}\right)\\
& \ll L \sum_{d>D_2} \left(  1 + \frac{\left(\frac{B}{d}\right)^{n+1}}{L^{n+1}R^n} +  \frac{\left(\frac{B}{d}\right)}{L}\right) \\
&\ll BL + B \log B + \frac{B^{n+1}}{L^nR^n} D_2^{-n}.
\end{align*}
For $n\geq 4$, we have 
\begin{equation*}
    BL + B\log B \ll B^{1+\f{1-\tau}{2}} \ll B^{\l(\f{1 + \tau}{2}\r)(n-1)- 2\tau} \ll \f{B^{n-1-2\tau}}{(LR)^{n-1}}\ll \f{B^{n-1}}{(LR)^{n-1}} E_{\tau}(B),
\end{equation*}
and 
\begin{equation*}
    \f{B^{n+1}}{L^nR^n}D_2^{-n} \ll \f{B^{n-1}}{(LR)^{n-1}}\cdot \f{B^2}{LR}\cdot \l(B^{\f{1+\tau}{2}-\eps}\r)^{-n} \ll \f{B^{n-1 + \eps}}{(LR)^{n-1}}B^{2-\f{n}{2}} \ll \f{B^{n-1 + \eps}}{(LR)^{n-1}} E_{\tau}(B).
\end{equation*}
Combining the estimates that we have obtained for three different ranges for $d$, we conclude the proof of Lemma \ref{le:dsum}.
\end{proof}

\begin{proof}[Completion of proof of Theorem \ref{thm:countingV}]
   In view of \eqref{eq:cir2} and the estimates in Proposition \ref{prop:growthofcirsigma}, it remains to analyse arithmetic part of the main term in Lemma \ref{le:dsum} and to show that it  equals $\omega^V_{f}\left(\CD(L,\bLambda)\right)$, the finite Tamagawa measure of finite adelic neighbourhood $\CD(L,\bLambda)$.
   As in the proof of Proposition \ref{prop:growthofcirsigma}, we fix $L_0\in\BN$ such that both $\CV$ and $\CW^o$ are regular modulo $L_0$, and upon decomposing the neighbourhood $\CD(L,\bLambda)$ into a disjoint union whose level of congruence are all divisible by $L_0$, we may assume that $L_0\mid L$. 

We have
	$$\mathfrak{S}(\CW;L,\gamma\overline{d}\bGamma)=\prod_{p\mid L}\sigma_{p}(\CW;L,\gamma\overline{d}\bGamma)\prod_{p\nmid L}\sigma_p(\CW).$$ For every $p\mid L$ (recalling \eqref{eq:Ep}, \begin{align*}
		\sigma_{p}(\CW;L,\gamma\overline{d}\bGamma)=\lim_{k\to\infty}\frac{\#\left(\CW^o(\BZ/p^k\BZ)\cap \CE_p(L,\gamma\overline{d}\bGamma)\right)}{p^{nk}}=\frac{1}{p^{nm_p}}.
	\end{align*}  In particular its value is independent of $\gamma$ and $d$.
	Therefore we obtain
	\begin{equation*}\label{eq:arithfactor}
		\begin{split}
		&\sum_{\gamma\in \BG_{\operatorname{m}}(\BZ/L\BZ)}\sum_{(d,L)=1}\frac{\mu(d)}{d^{n-1}}\mathfrak{S}(\CW;L,\gamma\overline{d}\bGamma)\\ =& \prod_{p\mid L}
		\left(1-\frac{1}{p}\right)\frac{1}{p^{(n-1)m_p}}\times\prod_{p\nmid L}\left(1-\frac{1}{p^{n-1}}\right)\sigma_p(\CW)\\ =&\prod_{p\mid L}\left(1-\frac{1}{p}\right)\sigma_{p}(\CV;L,\boldsymbol{\Lambda})\times\prod_{p\nmid L}\left(1-\frac{1}{p}\right)\sigma_p(V) =\omega^V_{f}\left(\CD(L,\bLambda)\right),
				\end{split}
	\end{equation*}
	on using \eqref{eq:singserV} \eqref{eq:sigmapV1} \eqref{eq:sigmapV2}.
 This finishes the proof. \end{proof}

	\appendix
	
	\section{Rational approximations on quadrics}\label{se:ratapprox}
	The goal of this appendix is to  first recall the ``$\alpha$-constant''
 introduced in the recent work \cite{M-R} of McKinnon-M. Roth, which pioneers the study of Diophantine approximation on algebraic varieties. (For later developments see also  \cite{Huang}.)
This is a generalisation of the classical notion of  ``irrationality measure'' originally attached to real numbers. We then compute its value achieved on generic open subsets on projective quadrics, so as to deduce the optimality of Theorem \ref{thm:countingV}.
	
	\subsection{Approximation constants} To ease the exposition, let us place ourselves in the following setting. Let $X\hookrightarrow \BP^n$ be an irreducible projective algebraic variety defined over $\BQ$. Let $L$ be an ample line bundle, to which we associate a height function $H_L$. 
	We choose a projective distance function $d_\nu$ on $\BP^n$ for a place $\nu$ of $\BQ$ as in \cite[\S2]{M-R}. Let $Q_0\in X(\BQ_\nu)$ be a fixed $\nu$-adic point.
	\begin{definition}\label{def:appconst}
		For any subvariety $Y\subset X$, we define the (\textit{best}) \textit{approximation constant} $\alpha^\nu(Q_0,Y)$ (depending also on $L$ and $\nu$)  to be the \textbf{infimum} of $\gamma>0$ such that there exist infinitely many points $P\in Y(\BQ)$ such that
  $$  d_\nu(P,Q_0)^\gamma H_L(P)<1.$$

	\end{definition}
	It can be shown that (see \cite[\S2]{M-R} for details) $\nu$-adically around $Q_0$ all distances are equivalent, hence the $\alpha$ constant does not depend on the choice of $d_\nu$, nor on the choice of $H_L$, nor is the bound $1$ of importance.
 The constant $\alpha^\nu(Q_0,Y)$ measures how quickly the height must grow for a sequence of rational points of $Y$ that tends to $Q_0$. The smaller the $\alpha$-constant is, the slower the height is allowed to grow, which means $Q_0$ is better approximable. In dimension one the approximation constant for algebraic points is known, as a corollary of the classical K. Roth theorem (\cite[Theorem 2.16]{M-R}) for $\BP^1$.
	
	We can furthermore ask for approximating the point $Q_0$ in a ``generic direction''. This motivates the definition of another constant as follows. 
	It was first formulated by Pagelot (unpublished) and further developed in the first author's Ph.D. thesis \cite{Huang}.
	\begin{definition}\label{def:essconst}
		With the notation above, we define the \textit{essential constant} of $Q_0$ 
		to be
		$$\aess^\nu(Q_0)=\sup_{\substack{Y\subset X}} \alpha^\nu(Q_0,Y),$$
		where $Y$ ranges over all open dense subvarieties of $X$. 
	\end{definition}
	
\subsection{The essential constants for quadrics}
	We go back to our initial setting. 	Let $F\in \BZ[x_0,\ldots,x_n]$ be an integral non-degenerate quadratic form. 	Let $V\subset\BP^n$ be the projective quadric defined by $F=0$. We always assume that $V$ is $\BQ$-isotropic, that is $V(\BQ)\neq\varnothing$. We shall use the naive height $H:\BP^n(\BQ)\to\BR_{>0}$ corresponding to $\CO(1)$ and the global sections given by the coordinates $x_0,\ldots, x_n$.

Theorem \ref{thm:countingV} has the following consequence on bounding the essential $\alpha$-constant from above.
	\begin{theorem}\label{thm:aess}
	Assume that $n\geqslant 4$ (or equivalently $\dim V\geqslant 3$). Then for every place $\nu$ of $\BQ$, and for every $Q_0\in V(\BQ_\nu)$, we have
	$$\aess^{\nu}(Q_0)\leqslant 2.$$
\end{theorem}
\begin{proof}
Fix some proper closed subvariety $Z \subseteq V$. 
Recall the real neighborhood $U_0 \subseteq \mathbb{R}^{n-1}$ and the neighborhood $\CU(R,Q_0)$ from \eqref{real zoom condition in x}, which is an open real neighbourhood of side-length $R^{-1}$ around $Q_0$. Pick a box $U_1 \subseteq U_0$ such that $g_{Q_0}^{-1}(R^{-1}U_1)$ is disjoint from $Z(\mathbb{R})$. Choose a smooth compactly supported weight function $w_4:\mathbb{R}^{n-1}\ra \mathbb{R}_{\geq 0}$ with support contained in $U_1$. Define similarly to equation (\ref{choice of wbr}) a weight function
\begin{equation}\label{choice of wbru_1}
	w_{B,R,U_1}^{\pm}(\bx) := w_1^{\pm}\left(R\frac{t_j}{t_0}, j \geq 2\right)w_2\left(\frac{R^2}{B^2}F(\bx)\right)w_3^{\pm}\left(\frac{t_0(\bx)}{B}\right)w_4\left(R\frac{t_j}{t_0}, j\geq 2\right),
\end{equation}
where we added an additional factor of $w_4$. Note that Theorem \ref{thm:countingW} still holds with the additional weight function $w_4$. For $B$ sufficiently large and for $0<\tau<1$, we find that
$$\CN_{\CW}(w_{B,R,U_1}^{\pm};(L,\bGamma)) \gg \frac{B^{n-1}}{L^nR^{n-1}},$$ whenever $LR=O(B^{\frac{1-\tau}{2}})$.
In particular, this quantity is positive, so we can find a point $P=P(B;R,L)$ which is counted by $\CN_{\CW}(w_{B,R,U_1}^{\pm};(L,\bGamma))$, and in particular $P\not\in Z(\BQ)$.  

Case $\nu=\BR$. We take $L=1$ and we omit $(L,\bLambda)$.
For any fixed $\tau \in (0,1)$, define $R = B^{\frac{1-\tau}{2}}$. 
Using the coordinate system $\bt$ as given in  Lemma \ref{change of variables}, the real zoom condition determined by the weight functions $w_1^\pm,w_4$ implies that  
$$d_{\BR}(P,Q_0)\ll \frac{1}{R}$$ for any $P$ as above and for any fixed choice of real distance functions around $Q_0$. Hence
$$d_{\BR}(P,Q_0)^{\frac{2}{1-\tau}}H(P)\ll \frac{B}{R^{\frac{2}{1-\tau}}}=1,$$ from which we infer $$\alpha^{\BR}(Q_0,V\setminus Z)\leqslant \frac{2}{1-\tau}.$$
With this we find
$$\aess^{\BR}(Q_0)\leqslant \inf_{0<\tau<1}2(1-\tau)^{-1}=2.$$

Case $\nu$ ultrametric.
We similarly take $R=1$ and $L=p^{m_p}$ where $p$ corresponds to $\nu$. Now the $p$-adic condition determined by the pair $(p^{m_p},Q_0)$ implies that $$d_{\nu}(P,Q_0)\ll p^{-m_p},$$ for any $P$ as above and for any fixed $p$-adic distance functions around $Q_0$. The arguments now follow similarly as in the archimedean case. \end{proof}

We now give the exact value of the essential constant of any rational point on $V$ via elementary arguments. This shows the optimality of Theorem \ref{thm:aess} (and hence Theorem \ref{thm:countingV}).
\begin{proposition}\label{prop:aessgeq}
	Assume that $n\geqslant 2$. Then for any place $\nu$ and for any $Q_0\in V(\BQ)$, we have $$\aess^{\nu}(Q_0)=2.$$
\end{proposition}
\begin{proof}
We shall prove that \begin{equation}\label{eq:loweralpha}
		\alpha^\nu(Q_0,U)\geqslant 2,
	\end{equation} for an appropriately chosen open subset $U\subset V$, which directly implies $\aess^{\nu}(Q_0)\geqslant 2$. In fact \eqref{eq:loweralpha} is a consequence of the following repulsion inequality  \begin{equation}\label{eq:repulsion}
	    \operatorname{d}_\nu(P,Q_0)^2 H(P)\geqslant A_{Q_0},
	\end{equation} valid for any $P\in U(\BQ)$ different from $Q_0$, where $A_{Q_0}>0$ depending only on $Q_0$. We are going to establish \eqref{eq:repulsion}.
 
 Upon a $\BQ$-linear transformation as in Lemma \ref{change of variables}, which clearly does not affect the value of $\alpha$ (see \cite[Proposition 2.4]{M-R}), we may assume that $Q_0=[1:0:\cdots:0]$ and $$F(x_0,\ldots,x_n)=x_0x_1+F_2(x_2,\ldots,x_n).$$ 
We take $\left(\frac{x_1}{x_0},\ldots,\frac{x_{n}}{x_0}\right)$ as a local chart for the open set $W:=(x_0\neq 0)$. 
We may assume that $\max_{1\leq i \leq n}|x_i/x_0| \ll 1$, else (\ref{eq:repulsion}) is trivial. By \cite[Proposition 2.4]{M-R} projective distance functions coming from different embeddings are equivalent.  So we may use the following projective distance function 
 $$\operatorname{d}_{\nu}(P,Q_0)=\max_{1\leqslant i\leqslant n}\left|\frac{x_i}{x_0}\right|_\nu,$$ whenever $P=[x_0:\cdots:x_n]\in W(\BQ)$. (In what follows, homogeneous coordinates of $P$ are all integers with $\gcd(x_0,\ldots, x_n)=1$.)
    We define the closed subset $V_1:=V\cap (x_1=0)$, and the open subset \begin{equation}\label{eq:U}
    U:=(W\cap V)\setminus V_1.
\end{equation}

We first consider the case where $\nu$ is real.
Since for every $P=[x_0:\cdots:x_n]\in U(\BQ)$, 
	$$|x_0|\leqslant |x_0x_1|=|F_2(x_2,\ldots,x_n)|\ll \max_{2\leqslant i\leqslant n}|x_i|^2,$$ we then have
	$$\operatorname{d}_{\BR}(P,Q_0)^2 H(P)\geqslant\left(\frac{\max_{2\leqslant i\leqslant n}|x_i|}{|x_{0}|}\right)^{2}|x_{0}|\gg \frac{|x_0|}{|x_0|^2}|x_0|\geqslant 1.$$
	
We next assume that $\nu$ corresponds to a prime $p$. 
For every $P=[x_0:\cdots:x_n]\in U(\BQ)\setminus\{Q_0\}$ with coprime integers $x_0,\ldots, x_n$, the equation $x_0x_1=-F_2(x_2,\ldots,x_n)$ implies that
$$\operatorname{ord}_p(x_1)+\operatorname{ord}_p(x_0)\geqslant 2\min_{2\leqslant i\leqslant n}\operatorname{ord}_p(x_i),$$ whence 
\begin{equation}\label{max12}
|x_0x_1|_\nu\leqslant \max_{2\leqslant i\leqslant n}|x_i|^2_\nu.
\end{equation}
Fix $i_0\geqslant 2$ (depending on $P$) such that $|x_{i_0}|^2_\nu$ attains the maximum above. We necessarily have $x_{i_0}\neq 0$ since otherwise  we would have $x_2=\cdots=x_n=0$ and this would imply $x_1=0$, contradicting to $P\not\in V_1(\BQ)$. 

Note that by equation (\ref{max12}) we find that $|x_1|_\nu^2\leq \max_{2\leq i\leq n}|x_i|_\nu^2$. We conclude that. $$\operatorname{d}_\nu(P,Q_0)^2 H(P)\geqslant \frac{|x_{i_0}|_\nu^2}{|x_0|_\nu^2}|x_1|\geqslant \frac{\max_{2\leqslant i\leqslant n}|x_i|^2_\nu}{|x_0x_1|_\nu}\geqslant 1.$$

In summary, we have proven that for every $P\in U(\BQ)\setminus\{Q_0\}$ and for every $\nu$, $$\operatorname{d}_\nu(P,Q_0)^2 H(P)\gg 1.$$ This confirms the claim \eqref{eq:repulsion}.

	Finally we prove the reverse inequality \begin{equation}\label{eq:upperaess}
	 	\aess^{\nu}(Q_0)\leqslant 2,
	 \end{equation} for arbitrary $\nu\in\operatorname{Val}(\BQ)$ and $Q_0\in V(\BQ)$. We shall exhibit a family of $\BQ$-curves through $Q_0$ which are  \emph{very free}, meaning that their union sweeping out a dense subset of the quadric. The idea essentially follows from the fact that quadrics containing at least one rational point are rational. We fix a rational hyperplane $H\subset \BP^n$ not containing $Q_0$. Then any generic line through $Q_0$  intersects $H$ at a unique point. This defines a $2$-to-$1$ projection map $\operatorname{pr}$ from an open subset of $V$ to $H$. Now the Zariski closure of $\operatorname{pr}^{-1}(l)$, the preimage of any generic line $l\subset H$ via $\operatorname{pr}$ in $V$, is a conic through $Q_0$. Moving the line $l\subset H$ gives a family of  smooth conics through $Q_0$, each one having $\alpha$-constant equal to $2$. This confirms the claim \eqref{eq:upperaess}.
\end{proof}

\begin{remarks}\label{rmk:dioapp}\hfill
\begin{enumerate}
    \item We can show that $\alpha^\nu(Q_0,V_1)=1$, whenever $Q_0\in V(\BQ)$ and $V_1(\BQ)\setminus\{Q_0\}\neq\varnothing$ (or equivalently $F_2$ is isotropic), where $V_1$ is defined as in the proof of Proposition \ref{prop:aessgeq}. So $V_1$ is ``locally accumulating'' in the sense of \cite{Huang}, which morally contains rational points that are ``better approximable'' to $Q_0$.
		
  \item Rational points are usually ``badly approximable'' compared to irrational ones, i.e., they have larger $\alpha$-constants. Recent work of de Saxcé \cite{Saxce} deals with rational approximations of a fixed real algebraic point on a projective quadric based on a version of Schmidt's subspace theorem, and it is proved that such a point has smaller  $\alpha$-constant. 
	
\item Although not directly related to the results of this article, it is worth pointing out that studying the distribution of rational points inside of a shrinking neighbourhood of size corresponding exactly to the essential constant, say $R\asymp B^{\frac{1}{2}}$, is an interesting (but challenging) problem. The limit distribution is in general quite different from the one given by the Tamagawa measures.
See \cite{Huang} and the references therein for more details on this subject.\footnote{Two examples of quadratic surfaces (split or not over $\BQ$) were addressed in Pagelot's thesis (unpublished). We are able to determine the shape of limit measures for higher dimensional quadratic hypersurfaces (unpublished).} 
\end{enumerate}
\end{remarks}

	\section*{Acknowledgements}
	We thank Miriam Kaesberg for her effort in an earlier stage of this project. We are grateful to Tim Browning for his interest. The hospitality of IST Austria, Universität Göttingen, and Bayes Centre at ICMS contributes substantially to the development of this paper. We thank these institutes heartily. The third author was supported by the University of Bristol and the Heilbronn Institute for Mathematical Research.

\end{document}